\documentclass[reqno]{amsart}
\usepackage{amsmath}
\usepackage{amsfonts}
\usepackage{amstext}
\usepackage{amsbsy}
\usepackage{amsopn}
\usepackage{amsxtra}
\usepackage{upref}
\usepackage{amsthm}
\usepackage{amsmath}
\usepackage{amssymb}

\usepackage[textures,bookmarks=false]{hyperref}
\usepackage{mathrsfs}

\newtheorem{prop}{Proposition}[section]
\newtheorem{rem}{Remark}[section]
\newtheorem{lema}{Lemma}[section]
\newtheorem{defi}{Definition}[section]
\newtheorem{teo}{Theorem}[section]

\newtheorem{claim}{Claim}
\newtheorem*{claim*}{Claim}

\newtheorem{maintheorem}{Theorem}

\def\eps{\varepsilon}
\def\phi{\varphi}
\def\R{{\mathbb R}}

\def\N{{\mathbb N}}

\def\E{{\mathcal E}}

\def\P{{\mathcal P}}

\def\F{{\mathcal F}}
\def\D{{\mathcal D}}
\def\M{{\mathcal M}}

\def\T{{\mathcal T}}

\def\es{{\emptyset}}
\def\sm{\setminus}

\def\crit{{\mathcal Cr}}

\def\bd{\partial }
\def\le{\leqslant}
\def\ge{\geqslant}
\def\st{such that }

\def\level{\text{lev}}
\def\F{\mathcal{F}}
\def\M{\mathcal{M}}
\def\cyl{{\rm C}}

\title[Natural equilibrium states for multimodal maps]{Natural equilibrium states for multimodal maps}
\date{\today}

\begin{thanks}
{GI was partially  supported by Proyecto Fondecyt 11070050 and by Research Network on Low Dimensional Systems, PBCT/CONICYT, Chile. MT is supported by FCT grant SFRH/BPD/26521/2006 and also by FCT through CMUP}
\end{thanks}

\subjclass[2000]{37D35, 37D25, 37E05}
\keywords{Equilibrium states, thermodynamic formalism, interval maps, non-uniform hyperbolicity}

\author{Godofredo Iommi} \address{Facultad de Matem\'aticas,
Pontificia Universidad Cat\'olica de Chile (PUC), Avenida Vicu\~na Mackenna 4860, Santiago, Chile}
\email{giommi@mat.puc.cl}
\urladdr{http://www.mat.puc.cl/\textasciitilde giommi/}
\author{Mike Todd} \address{Departamento de Matem\'atica Pura,
Faculdade de Ci\^encias da Universidade do Porto
Rua do Campo Alegre, 687, 4169-007 Porto, Portugal }
\email{mtodd@fc.up.pt}
\urladdr{http://www.fc.up.pt/pessoas/mtodd/ }

\begin{document}

\begin{abstract}
This paper is devoted to the study of the thermodynamic formalism for a class of real multimodal maps. This class contains, but it is larger than, Collet-Eckmann. For a map in this class, we prove existence and uniqueness of equilibrium states for the geometric potentials $-t \log|Df|$, for the largest possible interval of parameters $t$.
We also study the regularity and convexity properties of the pressure function, completely characterising the first order phase transitions. Results concerning the existence of absolutely continuous invariant measures with respect to the Lebesgue measure are also obtained.
\end{abstract}

\maketitle

\section{Introduction}

The class of dynamical systems whose ergodic theory is best understood is the class of \emph{hyperbolic dynamical systems}, or, more generally, systems where the interesting dynamics behaves in a uniformly hyperbolic way: Axiom A maps. This is due to several reasons, one of them is the fact that these systems often have a compact symbolic model whose dynamics is well known \cite{bo, Ruellebook}.  For real one-dimensional maps, Axiom A maps are defined to be the class of maps where all points are either uniformly expanded or map into an attracting basin. This class is large even within families of maps with critical points such as the quadratic family, in which case it is a dense set, see \cite{Lyu, GraSw}.  Note that these maps do have a compact symbolic model (see
\cite[Chapter 16]{kh}).  In the example of the quadratic family, maps which are not Axiom A are nowhere dense, but nevertheless have positive Lebesgue measure, see \cite{Jak, BenCar}.  Due to the rich dynamics of these systems, the expansion properties of such systems, can be very delicate.

In recent years a great deal of attention has been paid to non-Axiom A systems which are expanding on most of the phase space, but not in all of it.  The simplest example of these type of maps, namely \emph{non-uniformly hyperbolic} dynamical systems,  are interval maps with a parabolic fixed point (e.g. the Manneville-Pomeau map \cite{pm}).  The ergodic theory for these maps is fairly well understood  \cite{pm, PreSl, Sarphase} and qualitatively different from the one observed in the hyperbolic case.

We will study the ergodic theory of class of maps for which the lack of hyperbolicity  can be even stronger: interval maps with critical points. The techniques we develop are different from the ones used to study hyperbolic systems and systems with a parabolic fixed point.

In this paper we will be devoted to study a particular branch of ergodic theory, namely  \emph{thermodynamic formalism}. This is a set of ideas and techniques which derive from statistical mechanics \cite{bo, Kellbook, Ruellebook,Waltbook}. It can be thought of as the study of certain procedures for the choice of invariant measures. Let us stress that the dynamical systems we will consider have many invariant measures, hence the problem is to choose relevant ones. The main object in the theory is the topological pressure:

\begin{defi}
Let $f$ be an endomorphism
of a compact metric space and denote by $\M_f$ the set of $f-$invariant probability measures. Let $\phi: I \to [-\infty, \infty]$ be a  \emph{potential}.  Assuming that $\M_f\neq \es$, the \emph{topological pressure} of $\phi$ with respect to $f$ is defined, via the Variational Principle,  by
\begin{equation*}
P_f(\phi)=P(\phi) = \sup \left\{ h(\mu) + \int \phi \ d\mu :  \mu \in \M_f \textrm{ and } - \int \phi \ d\mu < \infty\right\},
\end{equation*}
where $ h(\mu)$ denotes the measure theoretic entropy of $f$ with respect to $\mu$.  We refer to the quantity in the curly brackets as the \emph{free energy of $\mu$} with respect to $(I, f, \phi)$.  Note that this is sometimes thought of as being minus the free energy; see for example \cite{Kellbook} for a discussion of this terminology.
\end{defi}

Note that we do not specify the regularity properties we require on the potential $\phi$.  If it is a continuous function, then the above definition coincides with classical notions of topological pressure (see \cite[Chapter 9]{Waltbook}).  In this paper we will be interested in the geometric potential $x\mapsto -t\log|Df(x)|$ for some parameter $t\in \R$. This function is continuous in the uniformly hyperbolic case, but is not  upper/lower semicontinuous for $t$ positive/negative for the class of dynamical systems that we will consider.

A measure $\mu_{\phi} \in  \M_f$ is called an \emph{equilibrium state}  for $\phi$ if it satisfies:
\[ h(\mu_{\phi}) + \int \phi \ d\mu_{\phi}= P(\phi). \]
In such a way, the topological pressure provides a natural way to pick up measures. Questions about existence, uniqueness and ergodic properties of equilibrium states are at the core of the theory. For instance, if the dynamical system $f$ is transitive, uniformly hyperbolic and the potential $\phi$ is
H\"older continuous then there exist a unique equilibrium state $\mu_{\phi}$ for $\phi$ and it has strong ergodic properties \cite{bo,Ruellebook}. Moreover, the hyperbolicity of the system is reflected on the regularity of the pressure function $t \mapsto P(t \phi)$. Indeed, the function is real analytic. When the system is no longer hyperbolic, as in the case of the Manneville-Pomeau map, then uniqueness of equilibrium states may break down \cite{PreSl}
and the pressure function might exhibit points where it is not analytic, the so called \emph{phase transitions} \cite{Sarphase}.

As mentioned above, we will consider maps for which the lack of hyperbolicity is strong: not only do the maps have critical points, but the orbit of these points can be dense.  We consider a family of real multimodal maps.
To be more precise the class of maps we will consider is defined as follows.
We say that an interval map $f$ is $C^{1+}$ if it is $C^1$ and the derivative $Df$ is $\alpha$-H\"older for some $\alpha>0$.  Let $\mathcal F$ be the collection of $C^{1+}$ multimodal interval maps $f:I \to I$, where $I=[0,1]$,  satisfying:

\newcounter{Lcount}
\begin{list}{\alph{Lcount})}
{\usecounter{Lcount} \itemsep 1.0mm \topsep 0.0mm \leftmargin=7mm}

\item the critical set $\crit = \crit(f)$ consists of finitely many critical points $c$ with critical order $1 < \ell_c < \infty$, i.e., there exists a neighbourhood $U_c$ of $c$ and a diffeomorphism $g_c:U_c \to g_c(U_c)$ with $g_c(c) = 0$
     $f(x) = f(c) + g_c(x)^{\ell_c}$;
\item $f$ has negative Schwarzian derivative, i.e., $1/\sqrt{|Df|}$ is convex; 
\item $f$ is topologically transitive on $I$;
\item $f^n(\crit)  \cap f^m(\crit)=\es$ for $m \neq n$.
\end{list}

Conditions c) and d) are for ease of exposition, but not crucial.  In particular, Condition c) excludes that $f$ has any attracting cycles, or homtervals (a homterval is an interval $U$ such that $U, f(U), f^2(U), \ldots$ are disjoint and the omega limit set is not a periodic orbit).
Condition d) excludes that one critical point is mapped onto another. If that happened, it would be possible to consider these critical points as a `block', but to simplify the exposition, we will not do that here.  Condition d) also excludes that critical points are preperiodic, a case which is easier to handle (for example by combining \cite[Chapter 16]{kh} and \cite{bo}) and does not require the theory we present here, see Section~\ref{sec:preper}.  Together c) and d) exclude the renormalisable case.

\begin{rem}
Usually in ergodic theory for one-dimensional dynamics it is assumed that the map is $C^2$.  A significant reason is that $C^2$ multimodal maps satisfying a) and b) have no homtervals and the non-wandering set $\Omega$ (the set of points $x\in I$ \st for arbitrarily small neighbourhoods $U$ of $x$ there exists $n(U)\ge 1$ \st $f^n(U)\cap U\neq \es$) can be broken down into finitely many elements $\Omega_k$, on each of which $f$ is topologically transitive, see \cite[Section III.4]{MSbook}.  However, for the maps we consider, assumption c) removes the necessity of a $C^2$ assumption.
We note that in the case where there is more than one transitive element in $\Omega$, for example the renormalisable case, the analysis presented in this paper can be applied to any one of the transitive elements consisting of a union of intervals permuted by $f$.

Now let $\Omega_{int}$ denote the union of all elements of $\Omega$ which consist of intervals permuted by $f$.  If, contrary to the assumptions on $\F$ above, $\Omega_{int}$ did not cover $I$ then there would be a (hyperbolic) Cantor set consisting of points which are always outside $\Omega_{int}$.  Dobbs \cite{Dobphase} showed that for renormalisable maps these hyperbolic Cantor sets can give rise to phase transitions in the pressure function not accounted for by the behaviour of critical points themselves.
\label{rmk:wandering set}
\end{rem}

\begin{rem}
The smoothness of our maps is important for two further reasons: to allow us to bound distortion on iterates, and to guarantee the existence of `local unstable manifolds'.  For the first, the tool we use is the Koebe Lemma, see \cite[Section IV]{MSbook}.  The negative Schwarzian condition we impose still allows us to use this for $C^{1+}$ maps.  For a detailed explanation of this issue see \cite{Cedthesis}.

Given a measure $\mu\in \M_f$, the existence of local unstable manifolds was used in \cite{BrCMP, BTeqgen} to show the existence of some natural `inducing schemes' (see Section~\ref{sec:ind schemes}).  As shown by Ledrappier \cite{Led}, and later generalised by Dobbs \cite{Dobcusp} (see the appendix), we only need a $C^+$ condition on $f$ to guarantee the existence of local unstable manifolds.
\label{rmk:c2 vs c1}
\end{rem}

Note that our class $\mathcal F$ includes transitive Collet-Eckmann maps, that is maps where $|Df^n(f(c))|$ grows exponentially fast.  Therefore the set of quadratic maps in $\mathcal F$ has positive Lebesgue measure in the parameter space of quadratic maps (see \cite{Jak, BenCar}).  In the appendix we show that our theory can be extended to a slightly more general class of maps, similar to the above, but only piecewise continuous.

As mentioned above, we will be particularly interested in the thermodynamic formalism for the geometric potentials $x\mapsto -t\log|Df|$.  The study of these potentials has various motivations, for example the relevant equilibrium states and the pressure function are related to the Lyapunov spectrum, see for example \cite{T}.
Moreover, important geometric features are captured by this potential. Indeed, in several settings, the equilibrium states for this family are associated to conformal measures on the interval. This allows the study of the fractal geometry of dynamically relevant subsets of the space. Moreover, by \cite{Led} any equilibrium state $\mu$ for the potential $x\mapsto -\log|Df|$ is an absolutely continuous invariant probability measure (acip) provided $\lambda(\mu)>0$.

For $\mu\in \M_f$, we define the \emph{Lyapunov exponent of $\mu$} as
\[\lambda(\mu):= \int \log |Df| \ d\mu. \]
We let
$$\lambda_M:=\sup\{\lambda(\mu):\mu\in \M_f\}, \ \lambda_m:=\inf\{\lambda(\mu):\mu\in \M_f\}.$$

\begin{rem}
Our assumptions on $f\in \F$, particularly non-flatness of critical points and a lack of attracting periodic cycles, means that by \cite{Prz}, $\lambda_m \ge 0$.
\label{rmk:prz}
\end{rem}

We let
$$p(t):=P(-t\log|Df|)$$
and define
\begin{equation} \label{eq:t plus minus}
t^-:=\inf\{t:p(t)>-\lambda_M t\} \text{ and }
t^+:=\sup\{t:p(t)>-\lambda_m t\}.
\end{equation}

Note that if $t^-\in \R$ (resp $t^+\in \R$) then $p$ is linear for all $t\le t^-$ (resp $t\ge t^+$).
We will later prove that for maps in $\F$, $t^-=-\infty$.
We prove in Proposition~\ref{prop:regular} that $t^+>0$.
In some cases $t^+=\infty$.  As we will show later, for non-Collet Eckmann maps with quadratic critical point, $\lambda_m=0$ and $t^+=1$.  \cite{MakSm} suggests that there should also be Collet-Eckmann maps with $t^+\in (1,\infty)$.  
In Proposition \ref{prop:pos press} we prove that under certain assumptions $t^+ \ge 1$: we expect that to be true for any map $f \in \F$.

The following is our main theorem.

\begin{maintheorem} \label{thm:eq_exist_unique}
For $f\in \F$ and $t \in (-\infty , t^+)$ there exists a unique equilibrium measure $\mu_t$ for the potential $-t \log |Df|$. Moreover, the measure $\mu_t$ has positive entropy.
\end{maintheorem}

A classical way to show the existence of equilibrium states is to use upper semicontinuity of entropy and the potential $\phi$ (see \cite[Chapter 4]{Kellbook}), and in particular the upper semicontinuity of $\mu\mapsto \int\phi~d\mu$.
However, in our setting even though, as noted in \cite{BrKell}, for $f\in \F$ the entropy map is upper semicontinuous, the existence of equilibrium measures in the above theorem is not guaranteed since the potential  $-t \log |Df|$ is not upper semicontinuous for $t>0$.  So for example, by \cite[Proposition 2.8]{BrKell} for unimodal maps satisfying the Collet-Eckmann condition, $\mu \mapsto -\lambda(\mu)$ is not upper semicontinuous.  Theorem~\ref{thm:eq_exist_unique} generalises \cite{BrKell} which applies to unimodal Collet-Eckmann maps for a small range of $t$ near 1; \cite{PeSe} which applies to a subset of Collet-Eckmann maps, but for all $t$ in a neighbourhood of $[0,1]$; and \cite[Theorem 1]{BTeqnat} which applies to a class of non-Collet Eckmann multimodal maps with $t$ in a left-sided neighbourhood of 1.

In order to prove Theorem~\ref{thm:eq_exist_unique} we use the theory of inducing schemes developed in \cite{BrCMP, BTeqgen, BTeqnat, T}.  Let us note that the thermodynamic formalism is understood for certain complex rational maps.  For example, Przytycki and Rivera-Letelier \cite{PrzRL} proved that if $f: \overline{\mathbb{C}} \to \overline{\mathbb{C}}$ is a  rational map of degree at least two, is expanding away from the critical points and has `arbitrarily small nice couples' then the pressure function $p$ is real analytic in a certain interval.  These conditions are met for a wide class of rational maps including topological Collet-Eckmann rational maps, any at most finitely renormalisable polynomial with no indifferent periodic orbits, as well as every real quadratic polynomial.

Related to the above are the regularity properties of the pressure function.

\begin{defi} \label{def:kink}
Let $\phi :[0,1] \to \R$ be a potential. The pressure function has a \emph{first order phase transition} at $t_0 \in \R$ if $p$ is not differentiable at $t=t_0$.  
\end{defi}
The pressure function is continuous and convex (see \cite[Theorem 9.7]{Waltbook}), which implies that the left and right derivatives $D^-p(t)$ and $D^+p(t)$ at each $t$ exist.  Moreover, the pressure, when finite, can have at most a countable number of points $t_i$ where it is not differentiable (i.e, $Dp^-(t_i)\neq D^+p(t_i)$), hence of first order phase transitions.  The regularity of the pressure is related to several dynamical properties of the system. For example, it has deep connections to large deviations \cite{el} and to different modes of recurrence \cite{Sarphase, Sarcrit}. In Section~\ref{sec:smooth and convex} we prove that the pressure  function  restricted to the interval $(-\infty, t^+)$ not only does not have first order phase transitions, but it is  $C^1$.

\begin{maintheorem}
For $f\in \F$, the pressure function $p$ is  $C^1$, strictly convex and strictly decreasing in $t \in (-\infty , t^+)$.  
\label{thm:smooth}
\end{maintheorem}


First order phase transitions are also related to the existence of absolutely continuous invariant probability measures. If $p(t)=0$ for all $t\ge 1$ and there is an acip, then the pressure function is not differentiable at $t=1$. This occurs for example if $f\in \F$ is unimodal and non-Collet Eckmann, but has an acip (see \cite{NoSa}). 
The following proposition gives the converse result.

\begin{prop} Let $f\in \F$ be such that $p(1)=0.$
If the pressure function has a first order phase transition at $t=1$ then the map $f$ has an acip. \label{prop:acip}
\end{prop}

We summarise some of the other results we present here for the potential $x\mapsto -t\log|Df(x)|$ in the simpler case of unimodal maps with quadratic critical point in the following proposition.

\begin{prop}
If $f\in \F$ is unimodal, non-Collet Eckmann and $\ell_c=2$ then $p$ is $C^1$, strictly convex and decreasing  throughout $(-\infty, 1)$ and $p(t)=0$ for all $t\ge 1$.  Moreover,
\begin{list}{$\bullet$}{\itemsep 0.2mm \topsep 0.2mm \itemindent -0mm \leftmargin=5mm}
\item[(a)]  if $f$ has no acip then $p$ is $C^1$ throughout $\R$;
\item[(b)]  if $f$ has an acip then $p$ has a first order phase transition at $t=1$.
\end{list}
\label{prop:collected results}
\end{prop}

The paper is organised as follows.  In Section~\ref{sec:CMS} we give an introduction to the theory of thermodynamic formalism for countable Markov shifts, which is primarily due to Sarig. In Section~\ref{sec:ind schemes} we give some preliminary results on inducing schemes, which will allow us to code any of our systems by a countable Markov shift.  In Section~\ref{sec:zero press} we show that the inducing schemes in Section~\ref{sec:ind schemes} have some of the properties which will allow us to produce equilibrium states for our systems.  In Section~\ref{sec:Gibbs integ} we prove the most technically complex part of our paper which gives us the existence of equilibrium states for our systems.  Section~\ref{sec:unique} gives details of the uniqueness of these equilibrium states which then allows us to prove Theorem~\ref{thm:eq_exist_unique} in Section~\ref{sec:main thm}.  In Section~\ref{sec:smooth and convex} we prove Theorem~\ref{thm:smooth} and in Section~\ref{sec:kinks acips} we prove Propositions~\ref{prop:acip} and \ref{prop:collected results}. In Section~\ref{sec:remarks} we discuss statistical properties of the measures constructed, the ergodic optimisation problem  and the case in which the critical points are preperiodic. Finally in the appendix we show how the results of this paper extend to a class of Lorenz-like maps, of the kind studied by Rovella \cite{Rovella} and Keller and St Pierre \cite{KellStP}.

Note that many of the results we quote in this paper are proved using the theory of Markov extensions introduced by Hofbauer.  To prove our main theorems it is not necessary to explain this theory in any detail since it is sufficient to quote results from elsewhere.  However, for a short description of this construction, see the appendix.

\emph{Acknowledgements:}
We would like to thank N.\ Dobbs for his comments on earlier versions of this paper which improved both the results and the exposition.  We would also like to thank H.\ Bruin and J.\ Rivera-Letelier for their useful remarks.  MT would also like to thank the mathematics department at PUC, where some of this work was done, for their hospitality.


\section{Preliminaries: countable Markov shifts}
\label{sec:CMS}

In this section we present the theory of countable Markov shifts: an extension of the finite case, and the relevant model for many non-uniformly hyperbolic systems, including maps in $\F$.

Let $\sigma \colon \Sigma \to \Sigma$ be a one-sided Markov shift
with a countable alphabet $S$. That is, there exists a matrix
$(t_{ij})_{S \times S}$ of zeros and ones (with no row and no column
made entirely of zeros) such that
\[
\Sigma=\{ x\in S^{\N_0} : t_{x_{i} x_{i+1}}=1 \ \text{for every $i
\in \N_0$}\},
\]
and the shift map is defined by $\sigma(x_0x_1 \cdots)=(x_1 x_2
\cdots)$. We say that $(\Sigma,\sigma)$ is a \emph{countable
Markov shift}.
We equip $\Sigma$ with the topology generated by the cylinder sets
\[ C_{i_0 \cdots i_n}= \{x \in
\Sigma : x_j=i_j \text{ for } 0 \le j \le n \}.\] Given a function
$\phi\colon \Sigma \to\R$, for each $n \ge 1$ we set
\[
V_{n}(\phi) = \sup \left\{|\phi(x)-\phi(y)| : x,y \in \Sigma,\
x_{i}=y_{i} \text{ for } 0 \le i \le n-1 \right\}.
\]
We say that $\phi$ has \emph{summable variations} if
$\sum_{n=2}^{\infty} V_n(\phi)<\infty$.  We will sometimes refer to $\sum_{n=2}^{\infty} V_n(\phi)$ as the \emph{distortion bound} for $\phi$.  Clearly, if $\phi$ has
summable variations then it is continuous.   We say that $\phi$ is \emph{weakly H\"older continuous} if $V_n(\phi)$ decays exponentially.  If this is the case then it has summable variations. In what follows we assume
$(\Sigma, \sigma)$ to be topologically mixing (see \cite[Section 2]{Sartherm} for a precise definition).

It is a subtle matter to define a notion of topological pressure for countable Markov shifts. Indeed, the classical definition  for continuous maps on compact metric spaces is based on the notion of
$(n,\eps)$-separated sets (see \cite[Chapter 9]{Waltbook}). This notion depends upon the metric of the space. In the compact setting, since all metrics generating the same topology are equivalent, the value of the pressure does not depend upon the metric. However, in non-compact settings this is no longer the case. Based on work of Gurevich \cite{Gutopent, Gushiftent}, Sarig \cite{Sartherm} introduced a notion of pressure for countable Markov shifts which does not depend upon the metric of the space and which satisfies a Variational Principle.
Let $(\Sigma, \sigma)$ be a topologically mixing countable Markov shift, fix a symbol $i_0$ in the alphabet $S$ and
let $\phi \colon \Sigma \to \R$ be a potential of summable variations.  We let
\begin{equation}
Z_n(\phi, C_{i_0}):=\sum_{x:\sigma^{n}x=x} \exp \left(S_n\phi(x)\right) \chi_{C_{i_{0}}}(x)
\label{eq:Zn}
\end{equation}
where $\chi_{C_{i_{0}}}$ is the characteristic function of the
cylinder $C_{i_{0}} \subset \Sigma$, and
$$S_n\phi(x):=\phi(x)+\cdots +\phi\circ\sigma^{n-1}(x).$$
Moreover, the so-called \emph{Gurevich pressure} of $\phi$ is defined by
\[
 P^G(\phi) := \lim_{n \to
\infty} \frac{1}{n} \log Z_n(\phi, C_{i_0}).
\]
Since $\sigma$ is topologically mixing, one can show that $P^G(\phi)$ does not depend on $i_0$.  We define
$$\M_\sigma(\phi):=\left\{\mu\in \M_\sigma:-\int\phi~d\mu<\infty\right\}.$$
If $(\Sigma, \sigma)$ is the full-shift on a countable alphabet then the Gurevich pressure coincides with the notion of pressure introduced by Mauldin and Urba\'nski \cite{MUifs}.
Furthermore, the following property holds (see \cite[Theorem 3]{Sartherm}):
\begin{prop}[Variational Principle] \label{prop:VarPri}
If $\phi: \Sigma \to \mathbb{R}$ has summable variations and $P^G(\phi)<\infty$ then
\begin{equation*}
P^G(\phi)= \sup \left\{ h_{\mu}(\sigma) +\int_\Sigma \phi \, d
\mu : \mu\in \M_\sigma (\phi)\right\}.
\end{equation*}
\end{prop}
Let us stress that the right hand side of the above inequality only depends on the Borel structure of the space and not on the metric. Therefore, a notion of pressure which is to satisfy the Variational Principle need not depend upon the metric of the space.

The Gurevich pressure also has the property that it can be approximated by its restriction to compact sets. More precisely
\cite[Corollary 1]{Sartherm}:

\begin{prop}[Approximation property] \label{prop:approx}
If $\phi: \Sigma \to \mathbb{R}$ has summable variations  then
\begin{equation*}
P^G( \phi) = \sup \{ P_{\sigma|K}( \phi) : K \subset \Sigma : K \ne \emptyset \text{ compact and } \sigma\text{-invariant}  \},
\end{equation*}
where $P_{\sigma|K}( \phi)$ is the classical topological pressure on
$K$.
\end{prop}

We consider a special class of invariant measures. We say that $\mu\in \M_\sigma$ is a \emph{Gibbs measure} for the
function $\phi \colon \Sigma \to \R$ if for some constants $P$,
$C>0$ and every $n\in \N$ and $x\in C_{i_0 \cdots i_n}$ we have
\begin{equation*} 
\frac{1}C \le \frac{\mu(C_{i_0\cdots i_n})}{\exp \left(-nP + S_n\phi(x)\right)} \le C.
\end{equation*}
 We refer to any such $C$ as a \emph{distortion constant} for the Gibbs measure.
It was proved by Mauldin and Urba\'nski \cite{muGIBBS} and by Sarig in \cite{SarBIP} that if $(\Sigma, \sigma)$ is a full-shift and the
function $\phi$ is of summable variations with finite Gurevich pressure $P^G(\phi)$ then it has an invariant Gibbs measure.  Moreover $P=P^G(\phi)$, and if $-\int\phi~d\mu<\infty$ then $\mu$ is an equilibrium state for $\phi$.  Furthermore, this is the unique equilibrium state for $\phi$ by \cite[Theorem 3.5]{muGIBBS} and \cite{BuSar}.

\section{Inducing schemes}
\label{sec:ind schemes}

In order to prove Theorem~\ref{thm:eq_exist_unique} we will use the machinery of inducing schemes.
We will use the fact that inducing schemes for the system $(I,f)$ can  be coded by  the full-shift on countably many symbols. 

Given $f\in \F$, we say that $(X,F,\tau)$ is an \emph{inducing scheme} for $(I,f)$ if

\begin{list}{$\bullet$}{\itemsep 0.2mm \topsep 0.2mm \itemindent -0mm \leftmargin=5mm}
\item $X$ is an interval containing a finite or countable
collection of disjoint intervals $X_i$ \st $F$ maps each $X_i$
diffeomorphically onto $X$, with bounded distortion  on all iterates (i.e. there
exists $K>0$ so that if there exist $i_0, \ldots, i_{n-1}$ and $x,y$ such that $F^j(x), F^j(y)\in X_{i_j}$ for $j=0, 1, \ldots, n-1$ then $1/K\le DF^n(x)/DF^n(y) \le K$); 
\item $\tau|_{X_i} = \tau_i$ for some $\tau_i \in \N$ and $F|_{X_i} = f^{\tau_i}$.  If $x \notin \cup_iX_i$ then $\tau(x)=\infty$.
\end{list}
The function $\tau:\cup_i X_i \to \N$ is called the {\em inducing time}. It may
happen that $\tau(x)$ is the first return time of $x$ to $X$, but
that is certainly not the general case.  For ease of notation, we will frequently write $(X,F)=(X,F,\tau)$.  We denote the set of points $x\in I$ for which there exists $k\in \N$ such that $\tau(F^n(f^k(x)))<\infty$ for all $n\in \N$ by $(X,F)^\infty$.

Given an inducing scheme $(X,F, \tau)$, we say that a  probability measure $\mu_F$ is a \emph{lift} of $\mu$ if for any $\mu$-measurable subset $A\subset I$,
\begin{equation} \mu(A) = \frac1{\int_X \tau \ d\mu_F} \sum_i \sum_{k = 0}^{\tau_i-1} \mu_F( X_i \cap f^{-k}(A)). \label{eq:lift}
\end{equation}
Conversely, given a measure $\mu_F$ for $(X,F)$, we say that
$\mu_F$ \emph{projects} to $\mu$ if \eqref{eq:lift} holds.  Note that if \eqref{eq:lift} holds then $\mu_F$ is $F$-invariant if and only if $\mu$ is $f$-invariant.  We call a measure
$\mu$  \emph{compatible with} the inducing scheme $(X,F,\tau)$ if

\begin{list}{$\bullet$}{\itemsep 1.0mm \topsep 0.0mm \leftmargin=5mm}
\item $\mu(X)> 0$ and $\mu\left(X \setminus (X,F)^\infty\right) = 0$; and
\item there exists a measure $\mu_F$ which projects to $\mu$ by
\eqref{eq:lift}: in particular $\int_X \tau \ d\mu_F <
\infty$ (equivalently $\mu_F\in \M_F(-\tau)$).
\end{list}

\begin{rem}
Given an ergodic measure $\mu\in \M_f$ with positive Lyapunov exponent there exists an inducing scheme $(X, F, \tau)$ with a corresponding $F-$invariant measure $\mu_F$, see for example \cite[Theorem 3]{BTeqnat}.
\end{rem}

\begin{defi}
Let $(X, F, \tau)$ be an inducing scheme for the map $f$.  Then for a potential $\phi:I\to \R$,  the induced potential $\Phi$ for $(X,F, \tau)$ is given by $$\Phi(x)=\Phi^F(x):=S_{\tau(x)}\phi(x).$$
\end{defi}

Note that in particular for the potential $\log|Df|$, the induced potential for a scheme $(X,F)$ is $\log|DF|$. Moreover, the map $x\mapsto \log|DF(x)|$ has summable variations (see for example \cite[Lemma 8]{BTeqnat}).

\begin{rem} \label{rmk:conj to shift}
Let $(X, F, \tau)$ be some inducing scheme for the map $f$.  We suppose that $\bd X\notin (X,F)^\infty$.  Then the system $F:(X,F)^\infty \to (X,F)^\infty$ is topologically conjugated to the full-shift on a countable alphabet.
\end{rem}

For an inducing scheme $(X, F,\tau)$ and a potential $\phi:X \to [-\infty, \infty]$ with summable variations, we can define the Gurevich pressure as in Section~\ref{sec:CMS}, and denote it by
$$P_F^G(\phi),$$
where we drop the subscript if the dynamics is clear.

In fact the domains for the inducing schemes used above come from the natural cylinder structure of the map $f\in \F$.  More precisely, the domains $X$ are $n$-cylinders coming from the so-called \emph{branch partition}: the set $\P_1^f$ consisting of maximal intervals on which $f$ is monotone.  So if two domains $\cyl_1^i, \cyl_1^j\in \P_1^f$ intersect, they do so only at elements of $\crit$. The set of corresponding $n$-cylinders is denoted $\P_n^f:=\vee_{k=1}^nf^{-k}\P_1$.  We let $\P_0^f:=\{I\}$.  For an inducing scheme $(X,F)$ we use the same notation for the corresponding $n$-cylinders $\P_n^F$.  Note the transitivity assumption on our maps $f$ implies that $\P_1^f$ is a generating partition for any Borel probability measure.



\section{Zero pressure schemes} \label{sec:zero press}

For $t\in \R$, we let $$\psi_t:=-t\log|Df|-p(t).$$  Similarly, for an inducing scheme $(X,F)$ the induced potential is $\Psi_t$.
As in \cite{PeSe, BTeqgen, BTeqnat} in order to apply Sarig's theory we need to find an inducing scheme $(X, F, \tau)$ so that $P^G(\Psi_t)=0$.  Then Sarig's theory gives a Gibbs measure for $(X, F, \Psi_t)$, which if it projects to a measure in $\M_f$ by \eqref{eq:lift}, must be an equilibrium state by the Abramov formula.  The main purpose of this section is to show that there are inducing schemes with $P^G(\Psi_t)=0$.

We note that a major difficulty when working with inducing schemes is that, in general, no single inducing scheme is compatible with all measures of positive Lyapunov exponent. As a direct consequence of work by Bruin and Todd \cite[Remark 6]{BTeqnat} we obtain in Lemma \ref{lem:bdd ind int} that there exists a finite number of inducing schemes for which any measure of entropy bounded away from zero is compatible with one of them. This will allow us to prove that for each $t\in (t^-,t^+)$  there exists an inducing scheme for which $P(\Psi_t)=0$ and such that the pressure, $p(t)$, can be approximated with $f$-invariant measures of positive entropy compatible with the inducing scheme.

\begin{prop}
For each $t\in (t^-,t^+)$, there exist an inducing scheme $(X,F)$
and a sequence $(\mu_n)_n\subset \M_f$ all compatible with $(X,F)$ and such that $$h(\mu_n)-t\lambda(\mu_n)\to p(t) \text{ and } \inf_n h(\mu_n)>0.$$  Moreover, $P^G(\Psi_t)=0$.
\label{prop:zero Gur}
\end{prop}

We need some lemmas and a definition for the proof.

\begin{lema} \label{lem:press less zero}
For each $t \in \R$ and any inducing scheme $(X,F)$, we have $P^G(\Psi_t)\le 0.$
\end{lema}

\begin{proof}
We let  $(X^N,F_N, \tau_N)$ denote the subsystem of $(X,F, \tau)$ where $X^N=\cup_{n=1}^NX_n$ and $F_N, \tau_N$ are the restrictions of $F, \tau$ to $X^N$.  Similarly, $P_{F_N}^G(\Psi_t)$ is defined in the obvious way.  By  Proposition \ref{prop:approx}, $P_F^G(\Psi_t)>0$ implies that for large enough $N$, $P_{F_N}^G(\Psi_t)>0$.  Hence there is an equilibrium state $\mu_{F_N}$ for this system so that  $\int\tau_N~d\mu_{F_N}<\infty$ and $$h(\mu_{F_N})-t\int \log|DF|~d\mu_{F_N}-p(t) \int \tau_N~d\mu_{F_N}>0.$$  Similarly to the use of the Abramov formula above, the corresponding projected measure $\mu_{f_N}$ as in \eqref{eq:lift} has
\[  h(\mu_{f_N}) -t \int \log |Df|~d\mu_{f_N}> p(t).    \]
This contradiction to the Variational Principle proves the lemma.
\end{proof}

\begin{rem}
By \cite[Lemma 8]{BTeqnat}, the potentials $\Psi_t$ we consider for the inducing schemes $(X,F)$ in Lemma~\ref{lem:bdd ind int} are weakly H\"older continuous.
\label{rmk:weak Holder}
\end{rem}

\begin{defi}
Given a function $g:[a,b] \to \R$, for $x_0\in \R$,  as in \cite[p115]{Royden}, we refer to $s:[a,b] \to \R$ as a \emph{supporting line for $g$ at $x_0$} if $s(x)=g(x_0)+p(x-x_0)$ for some $p\in \R$, and $g(x)\ge s(x)$ for all $x\in \R$.
\end{defi}

\begin{lema}
For each $t\in (t^-,t^+)$, there exists $\eta>0$ such that any measure $\mu$ with free energy w.r.t. $\psi_t$ close enough to 0 has $h(\mu)>\eta$. \label{lem:pos ent}
\end{lema}

\begin{proof}
Let $t_0 \in (t^-, t^+)$.
Suppose, by contradiction, that there exists a sequence of invariant measures $\mu_n$ such that $\lim_{n \to \infty} h(\mu_n)=0$ and
$$p(t_0)=\lim_{n \to \infty} \left( h(\mu_n) -t_0\lambda(\mu_n)   \right).$$
Denote by $L(t)$ the line passing through the origin which corresponds to the limit
$$\lim_{n \to \infty}  h(\mu_n) -t  \lim_{n \to \infty} \lambda(\mu_n).$$
We have that $L(t)$ is  a supporting line of the pressure $p(t)$ and $p(t_0)= L(t_0)$.
 Since $\lim_{n \to \infty} h(\mu_n)=0$, for every $t \ge t_0$ we have $p(t)=L(t)$. But this implies that $t_0= t^+$. This contradiction proves the statement.
\end{proof}
\begin{lema}
For each $\epsilon>0$ there exists $\theta>0$ and a finite number of inducing schemes $\{(X^n, F_n, \tau_n)\}_{n=1}^N$ such that any ergodic measure with $h(\mu)>\epsilon$ is compatible with one of these schemes $(X^n, F_n, \tau_n)$ and $\int\tau_n~d\mu_{F_n}<\theta$.
\label{lem:bdd ind int}
\end{lema}

\begin{proof}
This follows from \cite[Remark 6]{BTeqnat}.  We give a brief sketch of the ideas there.  That remark gives, for $\epsilon>0$, a set $\{(X^n, F_n, \tau_n)\}_{n=1}^N$ such that for each $\mu\in \M_f$ with $h(\mu)>\epsilon$, $\mu$ must be compatible with some $(X^n, F_n, \tau_n)$.  These schemes are constructed from sets $\hat X^n$ on the so-called Hofbauer extension (see the appendix for details).  The map $F$ is derived from a first return map $\hat F$ in this tower.  Measures $\mu\in \M_f$ with $h(\mu)>0$ can be lifted to the tower, and if they have $h(\mu)>\epsilon$ they must give  one of the sets $\hat X_n$ mass greater than some $\eta=\eta(\epsilon)>0$.  Since $\hat F$ is a first return map with return time $\hat\tau_n$, we use Kac's lemma to get $$\int\tau_n~d\mu =\int\hat\tau_n~d\hat\mu=\hat\mu(\hat X_n)^{-1}<\eta^{-1},$$
as required.
\end{proof}

As in \cite[Remark 6]{BTeqnat}, we denote this set of inducing schemes by $Cover(\epsilon)$.

\begin{proof}[Proof of Proposition~\ref{prop:zero Gur}]
By Lemmas~\ref{lem:bdd ind int} and \ref{lem:pos ent}, we can take a sequence of ergodic measures $\mu_p$ such that $$h(\mu_p)+\int\psi_t~d\mu_p =\epsilon_p \text{ where } \epsilon_p \to 0 \text{ as } p\to \infty,$$ $h(\mu_p)>\eta$ (some $\eta>0$), all $\mu_p$  are compatible with some inducing scheme $(X,F, \tau)\in Cover(\epsilon)$ and $\int\tau~d\mu_p<\theta$  for all $p\in \N$.  This implies that $P^G(\Psi_t)\ge 0$ since we have a sequence of measures $\mu_{F,p}$ such that
$$h(\mu_{F,p})+\int\Psi_t~d\mu_{F,p} = \left(\int\tau~d\mu_{F,p}\right)\left(h(\mu_{p})+\int\psi_t~d\mu_{p}\right) \ge \theta\epsilon_p$$  

On the other hand $P^G(\Psi_t)\le 0$ by Lemma~\ref{lem:press less zero}.  So the proposition is proved.
\end{proof}

Since the inducing scheme $(X,F)$ can be coded by the full-shift on countably many symbols we have, as explained in Section~\ref{sec:CMS}, a Gibbs measure $\mu_{\Psi_t}$ for $\Psi_t$.   We need to show that this measure has  integrable inducing time and thus that it projects to a measure in $\M_f$.

\section{The Gibbs measure has integrable inducing times}
\label{sec:Gibbs integ}

This section is devoted to proving that the inducing time is  integrable with respect to the Gibbs measure constructed in Section \ref{sec:zero press}. In particular, this implies that the measure has finite entropy and that it is an equilibrium state for the induced potential.
It also implies that it can be projected to a measure in $\M_f$.

\begin{prop}
Let $t\in (t^-, t^+)$ and $\psi=\psi_t$.  Suppose that we have an inducing scheme $(X,\tilde F)$.
Then there exists $k\in \N$ such that replacing $(X, \tilde F)$ by $(X,F)$, where $F=\tilde F^k$, the following holds.  There exist $\gamma_0\in (0,1)$ and, for any cylinder $\cyl_n^j\in \P_n^F$ any $n\in \N$, a constant $\delta_{n}^j<0$ such that any measure $\mu_F\in \M_F$ with $$\mu_F(\cyl_n^j)\le (1-\gamma_0) m_\Psi(\cyl_n^j) \text{ or } \mu_F(\cyl_n^j) \ge \frac{m_\Psi(\cyl_n^j)}{1-\gamma_0},$$ where $m_\Psi$ denotes the conformal measure for the system $(X,F,\Psi)$, must have $h(\mu_F)+\int\Psi~d\mu_F\le \delta_n^j$. \label{prop:conv to Gibbs}
\end{prop}

Note that $\delta_{n}^j \to 0$ as $m_\Psi(\cyl_n^j) \to 0$.  Also note that 
if  $K=\exp\left(\sum_{k=1}^\infty V_k(\tilde\Psi)\right)$ is a distortion constant for the potential $\tilde\Psi$ for the inducing scheme $(X, \tilde F)$ then it is also a distortion constant the potential $\Psi$ on $(X,F)$.

The following lemma will allow us to choose $k$ in the proof of Proposition~\ref{prop:conv to Gibbs}.  It is true for $\Psi=\Psi_t$, but also for more general potentials of summable variation.



\begin{lema}
Suppose that we have an inducing scheme $(X,F)$ and potential $\Psi=\Psi_t$ with distortion constant $K=\exp\left(\sum_{k=1}^\infty V_k(\Psi)\right)$ and $P^G(\Psi)=0$.  We let $m_\Psi$ denote the conformal measure for the system $(X,F,\Psi)$.  Then for any $\cyl_n\in \P_n^F$ and $n\in \N$, $$m_\Psi(\cyl_n)\le e^{-\lambda n}$$ where $\lambda:=-\log\left(K\sup_{\cyl_1\in \P_1^F} m_\Psi(\cyl_1)\right)$.
\label{lem:shrinking cylinders}
\end{lema}

\begin{proof}
Since $m_\Psi$ is a conformal measure, for $\cyl_n^i\in \P_n^F$ we have
$$ 1=m_\Psi(F^n(\cyl_n^i))=\int_{\cyl_n^i}e^{-S_n\Psi}~dm_\Psi.$$ So by the Intermediate Value Theorem we can choose $x\in \cyl_n^i$  so that $e^{S_n\Psi(x)}=m_\Psi(\cyl_n^i)$.  For future use we will write $S_n^i\Psi:=S_n\Psi(x)$.  Therefore, $$m_\Psi(\cyl_n^i)=e^{S_n^i\Psi}\le e^{n\sup\Psi}.$$
By the Gibbs property, $$e^{\sup\Psi} \le K \sup_{\cyl_1\in \P_1^F}m_\Psi(\cyl_1).$$   Therefore
$$\sup\Psi \le  \log\left(K\sup_{\cyl_1\in \P_1^F}m_\Psi(\cyl_1)\right).$$ We can choose this as our value for $-\lambda$.
\end{proof}


In the following proof we use the notation $A=\theta^{\pm C}$ to mean $\theta^{-C} \le A \le \theta^C$.

\begin{proof}[Proof of Proposition~\ref{prop:conv to Gibbs}]
Suppose that the distortion of the potential $\tilde\Psi$ for the scheme $(X,\tilde F)$ is bounded by $K\ge 1$.
We first prove that measures giving cylinders very small mass compared to $m_\Psi$ must have low free energy.  Note that for any $k\in \N$, the potential $\Psi$ for the scheme $(X,F)$ where $F=\tilde F^k$ also has distortion bounded by $K$. We will choose $k$ later so that $\lambda=\lambda(K,\sup_im_\Psi(X_i))$ for $(X,F)$ as defined in Lemma~\ref{lem:shrinking cylinders}, is large enough to satisfy the conditions associated to \eqref{eq:flat estimate}, \eqref{eq:sharp estimate} and \eqref{eq:gamma sharp}.  
Note that as in \cite[Lemma 3]{Sarphase} we also have $P^G(\Psi)=0$.

In Lemma~\ref{lem:flat low FE} below, we will use the Variational Principle to bound the free energy of measures for the scheme which, for some $\gamma$, have $\mu(\cyl_n^i)\le K m_\Psi(\cyl_n^i)(1-\gamma)/(1-m_\Psi(\cyl_n^i))^n$ in terms of the Gurevich pressure. 
However, instead of using $\Psi$, which,  in the computation of Gurevich pressure weights points $x\in \cyl_n^i$ by $e^{\Psi(x)}$, we use a potential which weights points in $\cyl_n^i$ by $(1-\gamma) e^\Psi(x)$.  That is, we consider $(X,F, \Psi^\flat)$ where
\begin{equation*}
\Psi^\flat(x)= \begin{cases} \Psi(x)+\log(1-\gamma)& \text{if } x\in \cyl_n^i,\\
\Psi(x) &  \text{if } x\in \cyl_n^j,\text{ for } j\neq i.
\end{cases}
\end{equation*}

Firstly we will compute $P^G(\Psi^\flat)$.

\begin{lema}
$P^G(\Psi^\flat)= \log\left(1-\gamma m_\Psi(\cyl_n^i)\right).$ \label{lem:pres psi flat}
\end{lema}

\begin{proof}
We prove the lemma assuming that $n=1$ since the general case follows similarly. We will estimate $Z_j(\Psi^\flat, \cyl_1^i)$, where $Z_j$ is defined in \eqref{eq:Zn}.  The ideas we use are similar to those in the proof of Claim 2 in the proof of \cite[Proposition 2]{BTeqnat}.  As can be seen from the definition,
$$Z_j(\Psi^\flat, \cyl_1^i) = e^{\pm \sum_{k=0}^{j-1}V_k(\Psi)} \sum_{\cyl_j\in \P_j^F\cap \cyl_1^i} \ \sum_{\text{any } x\in \cyl_j} e^{S_j\Psi^\flat(x)}.$$

 As in the proof of Lemma~\ref{lem:shrinking cylinders}, the conformality of $m_\Psi$ and the Intermediate Value Theorem imply that for each $k$ there exists $x_{\cyl_1^k}\in \cyl_1^k$ such that $m_\Psi(\cyl_1^k) = e^{\Psi(x_{\cyl_1^k})}$.  For the duration of this proof we write $\Psi_k := \Psi(x_{\cyl_1^k})$.  As above, we have $e^{\Psi_i^\flat}:=(1-\gamma) e^{\Psi_i}$.  Therefore, $$\sum_i e^{\Psi_i^\flat}= 1-\gamma e^{\Psi_i}.$$

For each $\cyl_j\in\P_j^F$ and for any $k\in \N$, there exists a unique $\cyl_{j+1}\subset \cyl_j$ such that $F^j(\cyl_{j+1})=\cyl_1^k$.  Moreover, there exists $x_{\cyl_{j+1}}\in \cyl_{j+1}$ such that $F^j(x_{\cyl_{j+1}}) = x_{\cyl_1^k}$.  Then for $\cyl_j\subset \cyl_1^i$,
\begin{align*}
\sum_{\cyl_{j+1}\subset\cyl_j} e^{S_{j+1}\Psi^\flat(x_{\cyl_{j+1}})} &= e^{\pm V_{j+1}(\Psi)} e^{S_j\Psi^\flat(x_{\cyl_{j}})} \left(\sum_i e^{\Psi_i^\flat}\right)\\
&= e^{\pm V_{j+1}(\Psi)} e^{S_j\Psi^\flat(x_{\cyl_{j}})} (1-\gamma e^{\Psi_i}).
\end{align*}

Therefore,
$$Z_{j+1}(\Psi^\flat, \cyl_1^i) = (1-\gamma e^{\Psi_i}) e^{\pm\left(V_{j+1}(\Psi)+\sum_{k=0}^{j-1}V_k(\Psi)\right)} Z_j(\Psi^\flat, \cyl_1^i),$$
hence
$$Z_{j+1}(\Psi^\flat, \cyl_1^i) = (1-\gamma e^{\Psi_i})^j e^{\pm\sum_{k=0}^{j} (k+1)V_{k}(\Psi)}.$$
As in Remark~\ref{rmk:weak Holder}, $\Psi$ is weakly H\"older, so $\sum_{k=0}^{j} (k+1)V_{k}(\Psi)< \infty$.   Therefore we have $P^G(\Psi^\flat)= \log(1-\gamma e^{\Psi_i})= \log(1-\gamma m_\Psi(\cyl_1^i))$, proving the lemma.
\end{proof}

For the next step in the proof of the upper bound on the free energy of measures giving $\cyl_n^i$ small mass, we relate properties of $(X,F, \Psi)$ and $(X,F, \Psi^\flat)$.


\begin{lema}
$\M_F(\Psi)=\M_F(\Psi^\flat)$ and  for any $\cyl_n^i\in \P_n^F$ we have
\begin{align*}
&\sup\left\{h_F(\mu) +\int\Psi~d\mu:\mu\in \M_F(\Psi),\ \mu(\cyl_n^i)<\frac{K(1-\gamma)}{(1-m_\Psi(\cyl_n^i))^n} m_\Psi(\cyl_n^i) \right\}\\
&\hspace{5mm} \le \sup\left\{h_F(\mu) +\int\Psi^\flat~d\mu:\mu\in \M_F(\Psi^\flat),\ \mu(\cyl_n^i)<\frac{K(1-\gamma)}{(1-m_\Psi(\cyl_n^i))^n} m_\Psi(\cyl_n^i) \right\}\\
&\hspace{10mm}-\left[\frac{K(1-\gamma)\log(1-\gamma)}{ (1-m_\Psi(\cyl_n^i))^n} \right] m_\Psi(\cyl_n^i) \\
&\hspace{5mm} \le P^G(\Psi^\flat)- \left[\frac{K(1-\gamma)\log(1-\gamma)}{(1-m_\Psi(\cyl_n^i))^n} \right] m_\Psi(\cyl_n^i).
\end{align*}
\label{lem:flat low FE}
\end{lema}

Note that we can prove that the final inequality is actually an equality, but since we don't require this here we will not prove it.

\begin{proof}
The fact that $\M_F(\Psi)=\M_F(\Psi^\flat)$ is clear from the definition.

Suppose that $\mu\in \M_F(\Psi)$ and $\mu(\cyl_n^i)\le m_\Psi(\cyl_n^i) K(1-\gamma)/(1-m_\Psi(\cyl_n^i))^n$.
Then
\begin{align*}
\left(h_F(\mu)+\int\Psi~d\mu\right)- \left(h_F(\mu)+\int\Psi^\flat~d\mu\right)= \int\Psi-\Psi^\flat~d\mu\\
=\mu(\cyl_n^i)(-\log(1-\gamma))\le -\left[\frac{K(1-\gamma)\log(1-\gamma)}{(1-m_\Psi(\cyl_n^i))^n}\right] m_\Psi(\cyl_n^i),
\end{align*}
proving the first inequality in the Lemma.  The final inequality follows from the definition of pressure.
\end{proof}

Lemmas~\ref{lem:pres psi flat} and \ref{lem:flat low FE} imply that any measure $\mu_F$ with $\mu_F(\cyl_n^i)<K(1-\gamma)m_\Psi(\cyl_n^i)/(1-m_\Psi(\cyl_n^i))^n$ must have
\begin{align}
h(\mu_F)+\int\Psi~d\mu_F & \le P^G(\Psi^\flat)-\left[\frac{K(1-\gamma) \log(1-\gamma)}{(1-m_\Psi(\cyl_n^i))^n}\right] m_\Psi(\cyl_n^i) \label{eq:inter flat estimate}\\
&\le \log\left(1-\gamma m_\Psi(\cyl_n^i)\right) -\left[\frac{K(1-\gamma)\log(1-\gamma)}{(1-m_\Psi(\cyl_n^i))^n} \right]m_\Psi(\cyl_n^i).
\label{eq:flat estimate}
\end{align}

If $m_\Psi(\cyl_n^i)$ is very small then $\log\left(1-\gamma m_\Psi(\cyl_n^i)\right) \approx -\gamma m_\Psi(\cyl_n^i)$ and so  choosing $\gamma\in (0,1)$ close enough to 1 the above is strictly negative. By Lemma~\ref{lem:shrinking cylinders}, $m_\Psi(\cyl_n^i)<e^{-\lambda n}$ so $\cyl_n^i$ is small if $\lambda$ large.  Hence if $\lambda$ is sufficiently large then we can set  $\gamma=\tilde\gamma^{\flat}\in (0,1)$ so that
$$\log\left(1-\tilde\gamma^\flat e^{-\lambda n}\right)-\left[\frac{K(1-\tilde\gamma^\flat)\log(1-
\tilde\gamma^\flat)}{(1- e^{-\lambda n})^n}\right]e^{-\lambda n}$$
is strictly negative for all $n\in \N$.  This implies that \eqref{eq:flat estimate} with $\gamma=\tilde\gamma^{\flat}$ is strictly negative for any $\cyl_n^i\in \P_n^F$ and any $n$, so we set \eqref{eq:flat estimate} to be the value $\delta_{n}^{i,\flat}$.

For the upper bound on the free energy of measures giving $\cyl_n^i$ relatively large mass, we follow a similar proof, but with
\begin{equation*}
\Psi^\sharp(x)= \begin{cases} \Psi(x)-\log(1-\gamma)& \text{if } x\in \cyl_n^i,\\
\Psi(x) &  \text{if } x\in \cyl_n^j,\text{ for } j\neq i.
\end{cases}
\end{equation*}

Similarly to above, one can show that
$\M_F(\Psi)=\M_F(\Psi^\sharp)$ and
\begin{align*}
&\sup\left\{h_F(\mu) +\int\Psi~d\mu:\mu\in \M_F(\Psi),\ \mu(\cyl_n^i)> \frac{m_\Psi(\cyl_n^i)}{K (1-\gamma) \left[1+m_\Psi(\cyl_n^i)\left(\frac\gamma{1-\gamma}\right)\right]^n} \right\}\\
&\hspace{2mm} \le \sup\left\{h_F(\mu) +\int\Psi^\sharp~d\mu:\mu\in \M_F(\Psi^\sharp),\ \mu(\cyl_n^i)> \frac{m_\Psi(\cyl_n^i)}{K (1-\gamma) \left[1+m_\Psi(\cyl_n^i)\left(\frac\gamma{1-\gamma}\right)\right]^n} \right\}\\
&\hspace{10mm}+\frac{\log(1-\gamma)m_\Psi(\cyl_n^i)}{K (1-\gamma) \left[1+m_\Psi(\cyl_n^i)\left(\frac\gamma{1-\gamma}\right)\right]^n} \\
&\hspace{2mm}\le P^G(\Psi^\sharp)+ \frac{\log(1-\gamma)m_\Psi(\cyl_n^i)}{K (1-\gamma) \left[1+m_\Psi(\cyl_n^i)\left(\frac\gamma{1-\gamma}\right)\right]^n}.
\end{align*}

Moreover, we can show that
$$P^G(\Psi^\sharp) = \log\left(1+ m_\Psi(\cyl_n^i) \left(\frac{\gamma}{1-\gamma}\right)\right)\le m_\Psi(\cyl_n^i) \left(\frac{\gamma}{1-\gamma}\right).$$

Therefore, if $\mu(\cyl_n^i)> \frac{m_\Psi(\cyl_n^i)}{K (1-\gamma) \left[1+m_\Psi(\cyl_n^i)\left(\frac\gamma{1-\gamma}\right)\right]^n}$, we have
\begin{equation}h_F(\mu) +\int\Psi~d\mu \le m_\Psi(\cyl_n^i) \left(\frac{\gamma}{1-\gamma}\right)+\frac{\log(1-\gamma) m_\Psi(\cyl_n^i)}{K (1-\gamma) \left[1+m_\Psi(\cyl_n^i)\left(\frac\gamma{1-\gamma}\right)\right]^n}.
\label{eq:sharp estimate}
\end{equation}

If $\lambda$ is sufficiently large then we can choose $\gamma=\tilde\gamma^\sharp\in (0,1)$ so that this is strictly negative and can be fixed to be our value $\delta_n^{i, \sharp}$. This can be seen as follows: let and $\gamma=p/(p+1)$ for some $p$ to be chosen later.  Then the right hand side of \eqref{eq:sharp estimate} becomes
\begin{equation}
m_\Psi(\cyl_n^i)\left(p+1\right)\left[\frac{p}{p+1} -\frac{\log(p+1)}{K(1+p e^{-\lambda n})^n} \right].
\label{eq:p sharp}
\end{equation} If $\lambda$ is sufficiently large, then there exists some large $\lambda'\in (0,\lambda)$ such that $(1+p e^{-\lambda n})^n \le 1+ p e^{-\lambda'n}$ for all $n\in \N$. Hence with this suitable choice of $\lambda$ we can choose $p$ so that the quantity in the square brackets in \eqref{eq:p sharp} is negative for all $n$.  So we can choose $\delta_n^{i,\sharp}<0$ to be \eqref{eq:sharp estimate} with $\gamma=\tilde\gamma^\sharp$.

We let
\begin{equation}
\gamma^\sharp=1-(1-\tilde\gamma^\sharp)\left(1+e^{-\lambda n} \left(\frac{\tilde\gamma^\sharp}{1- \tilde\gamma^\sharp}\right)\right)^n.
\label{eq:gamma sharp}
\end{equation}
For appropriately chosen $\lambda$ this is in $(0,1)$.

We set $\gamma_0':=\max\{\gamma^\flat, \gamma^\sharp\}$ and for each $\cyl_n^i\in \P_n^F$ we let $\delta_n^i:=\max\{\delta_n^{i, \flat}, \delta_n^{i, \sharp}\}$.  The proof of the proposition is completed by setting $\gamma_0:=1-K(1-\gamma_0')$, which we may assume is in $(0,1)$.
\end{proof}

\begin{prop}
There exists an inducing scheme $(X,F)$ such that for $t\in (t^-, t^+)$ and $\psi=\psi_t$, any sequence of measures $(\mu_n)_n$ with $h(\mu_n)-\int\psi~d\mu_n \to 0$ as $n\to \infty$ has a limit measure $\mu_\psi$ which is an equilibrium state for $\psi$.
\label{prop:conv to eq}
\end{prop}

Note that $(X,F)$ and $(\mu_n)_n$ can be chosen as in Proposition~\ref{prop:zero Gur}.


\begin{proof}
By Proposition~\ref{prop:zero Gur}, we can find $\theta>0$, an inducing scheme $(X,\tilde F)$ and a sequence of measures $(\mu_n)_n$ with $h(\mu_n)+\int\psi~d\mu_n \to 0$ each compatible with $(X,\tilde F)$ and with $\int\tilde\tau~d\mu_{\tilde F,n}<\theta$.  Proposition~\ref{prop:zero Gur} also implies $P^G(\tilde \Psi_t)=0$.   Taking $F=\tilde F^k$ for $k$ as in Proposition~\ref{prop:conv to Gibbs}, that proposition then implies that there exists $K'>0$ such that for any $\cyl_k\in \P_k^F$, for all large enough $n$,
$$\frac1{K'} \le \frac{\mu_{F,n}(\cyl_k)}{e^{S_k\Psi(x)}} \le K'$$ for all $x\in \cyl_k$ (note that as in Proposition~\ref{prop:conv to Gibbs}, we can actually take $K'=K/(1-\gamma_0)$ where $K$ is the distortion bound for $\tilde\Psi_t$).
Note that $(\mu_{F,n})_n$ is tight (see \cite[Section 25]{Bill} for a discussion of this notion) and that any limit of the sequence $\mu_{F,\infty}$ must satisfy the Gibbs property with distortion constant $K'$.  By the uniqueness of Gibbs measures (\cite{BuSar}),  $\mu_{F,\infty}=\mu_\Psi$.   We now show that $\int\tau~d\mu_{\Psi}<\theta k$.  First note that $\int\tau~d\mu_{F,n}= \int\tilde\tau^k~d\mu_{\tilde F,n}<\theta k$.  For the purposes of this proof we let $\tau_N:=\min\{\tau,N\}$.  By the Monotone Convergence Theorem,
$$\int\tau~d\mu_{\Psi}=\lim_{N\to \infty}\int\tau_N~d\mu_{\Psi}\le \lim_{N\to \infty}\limsup_{n\to \infty} \int\tau_N~d\mu_{F,n} \le \theta k.$$
Thus we can project $\mu_\Psi$ to $\mu_\psi$ by \eqref{eq:lift}.

The fact that $\mu_\psi$ is a weak$^*$ limit of $(\mu_n)_n$ follows as in, for example \cite[Section 6]{FreiTo}.  The fact that we have a  uniform bound $\mu_{F,n}\left\{\tau \ge j\right\}\le \theta k/j$ for all $n\in \N$ is again crucial in proving this.

The Abramov formula implies that $$\int\Psi~d\mu_\Psi= \left(\int\tau~d\mu_\Psi\right)\left(\int\psi~d\mu_\psi\right)= \left(\int\tau~d\mu_\Psi\right)(\lambda(\mu_\psi)-p(t)).$$
Since $\lambda(\mu)\in [\lambda_m, \lambda_M]$ and both $p(t)$ and $\int\tau~d\mu_\Psi$ are finite, this implies that $-\int\Psi~d\mu_\Psi<\infty$ and hence $\mu_\Psi$ is an equilibrium state for $\Psi$.  Using the Abramov formula again we have that $\mu_\psi$ is an equilibrium state for $\psi$.
\end{proof}

\begin{rem}
Here we give an example of a way our setting can be changed so that the arguments in Proposition~\ref{prop:conv to Gibbs} and \ref{prop:conv to eq} fail.  In the case where $f$ is the (appropriately scaled) quadratic Chebyshev polynomial, $t^-\in (-\infty, 0)$.  In this case there is a periodic point $p$ such that the Dirac measure $\delta_p$ on the orbit of $p$ has $\lambda(\delta_p)=\lambda_M$.
The point $p$ is the image of the critical point which means that our class of inducing schemes can not be compatible with $\delta_0$ (indeed the only inducing scheme for $\delta_0$ has only one domain and the only measure compatible to it is $\delta_0$).

However, any measure $\mu\in \M_f$ orthogonal to $\delta_0$ must have $h(\mu)-t\lambda(\mu) \le h(\mu_1)-t\lambda(\mu_1)$ for all $t\in \R$ where $\mu_1$ is the acip.  In particular, $h(\mu)-t\lambda(\mu)<p(t)$ for $t<t^-$.  If $P^G(\Psi_t)=0$ then arguments similar to those in the proofs of Lemma~\ref{lem:press less zero} and Proposition~\ref{prop:zero Gur} imply that there are measures with free energy w.r.t. $\psi_t$ is arbitrarily close to zero and positive entropy.  This contradiction implies that for $t<t^-$, $P^G(\Psi_t)<0$ so we cannot begin to apply the arguments above to that case.  So it is important that $t\in (t^-, t^+)$.
\end{rem}

\section{Uniqueness of equilibrium states}
\label{sec:unique}

The result in Proposition~\ref{prop:conv to eq} gives the existence of equilibrium states for $-t\log|Df|$ for each $t\in (t^-, t^+)$.  In this section we obtain uniqueness.  To do this we will use more properties of the inducing schemes described in \cite{BTeqnat}.  They were produced in as first return maps to an interval in the so-called Hofbauer tower.  This theory was further developed in \cite{BTeqgen} and \cite{T}.   The following theorem gives some of their properties.

\begin{teo}\label{thm:schemes}
There exists a countable collection $\{(X^n,F_n)\}_n$ of inducing schemes with $\bd X^n \notin (X^n,F_n)^\infty$ such that:
\newcounter{Mcount}
\begin{list}{\alph{Mcount})}{\usecounter{Mcount} \itemsep 1.0mm \topsep 0.0mm \leftmargin=5mm}
\item any ergodic invariant probability measure $\mu$ with $\lambda(\mu)>0$ is compatible with one of the inducing schemes $(X^n, F_n)$.  In particular there exists and ergodic $F_n$-invariant probability measure $\mu_{F_n}$ which projects to $\mu$ as in \eqref{eq:lift};
\item any  ergodic equilibrium state for $-t\log|Df|$ where $t\in \R$ with $\lambda(\mu)>0$ is compatible with all inducing schemes $(X^n, F_n)$.
\end{list}
\end{teo}

\begin{rem}
Note that it is crucial in our applications of Theorem~\ref{thm:schemes}, for example in the proofs of Proposition~\ref{prop:unique} and Proposition~\ref{prop:t^-}, that in b) we are able to weaken the condition $h(\mu)>0$ to $\lambda(\mu)>0$ when we wish to lift measures.  This is why we need to use a countable number of inducing schemes in Theorem~\ref{thm:schemes} rather than the finite number in \cite[Remark 6]{BTeqnat}.
\label{rmk:LE schemes}
\end{rem}

Before proving Theorem~\ref{thm:schemes}, we prove the following easy lemma.

\begin{lema} \label{lem:press zero}
If  $t \in (t^-, t^+)$ and an equilibrium state $\mu_t$ from Proposition~\ref{prop:conv to eq} is compatible with an inducing scheme $(X,F)$, then $P^G(\Psi_t)=0$.  Moreover the lifted measure $\mu_{t,F}$ is a Gibbs measure and an equilibrium state for $\Psi_t$.
\end{lema}

\begin{proof}
First note that by Lemma~\ref{lem:press less zero}, $P^G(\Psi_t)\le 0$.

Denote by $\mu_t$ an equilibrium measure for the potential $-t \log |Df|$ of positive Lyapunov exponent and let $\mu_{t,F}$ be the lifted measure.
Note that by Proposition~\ref{prop:VarPri} and by the Abramov formula, see for example \cite[Theorem 2.3]{PeSe}, we have
\begin{align*}
P^G\left(\Psi_t\right)  &\ge h(\mu_{t,F})+\int\Psi_t~d\mu_{t, F} =h(\mu_{t,F}) -t \int \log|DF| \ d \mu_{t,F} -p(t) \int \tau \ d \mu_{t,F} \\
&= \left(\int \tau \ d \mu_{t,F}\right) \left( \frac{h(\mu_{t,F})}{\int \tau \ d \mu_{t,F}} -t\left( \frac{\int \log|DF|~d \mu_{t,F}}{\int \tau~d \mu_{t,F}}\right)  -p(t) \right) \\
&=\left(\int \tau \ d \mu_{t,F}\right) \left( h(\mu_t) -t \int \log |Df| \ d \mu_t -p(t) \right).\end{align*}
But recall that $\mu_t$ is an equilibrium measure:
\[  p(t) = h(\mu_t) -t \int \log |Df| \ d \mu_t.\]
Therefore $P^G(\Psi_t) \ge 0$.

Since $P^G(\Psi_t)=0$ there exists a unique Gibbs measure $\mu_F$ corresponding to  $(X,F, \Psi_t)$.  By the Abramov formula,
$$h(\mu_{t,F})+\int \Psi_t~d \mu_{t,F}=0,$$ so $\mu_{t, F}$ is an equilibrium state for $(X,F, \Psi_t)$.  Since, in this setting,  equilibrium states are unique (see \cite{BuSar}) we have that $\mu_{t,F}=\mu_F$.
 \end{proof}

\begin{proof}[Proof of Theorem~\ref{thm:schemes}]
Part (a) of the theorem follows from the proof of \cite[Theorem 3]{BTeqnat}.  Part (b) is proved similarly to \cite[Proposition 2]{BTeqnat}, but with added information from our Proposition~\ref{prop:conv to Gibbs}.  We sketch some details. Suppose that $\mu$ is compatible to $(X^n, F_n)$.  Then Lemma~\ref{lem:press zero} implies that $P^G(\Psi_n)=0$.  Claim 1 of the proof of \cite[Proposition 2]{BTeqnat} implies that for any other inducing scheme $(X^{n'}, F_{n'})$ is `topologically connected' to $(X^n, F_n)$.  Proposition~\ref{prop:conv to Gibbs}, which is an improved version of Claim 2 in the proof of \cite[Proposition 2]{BTeqnat}, then can be used as in that proof to give a `metric connection' which means that an equilibrium state compatible with $(X^n,F_n)$ must be compatible with $(X^{n'}, F_{n'})$.
\end{proof}

\begin{prop}
For any $t\in (-\infty,t^+)$ there is at most one equilibrium state  for $-t\log|Df|$.  Moreover, if $t^+>1$ then for any $t\in \R$  there is at most one equilibrium state  for $-t\log|Df|$. \label{prop:unique}
\end{prop}

Clearly the equilibrium states, when unique, must be ergodic.
\begin{proof}
The idea here is first to show that any equilibrium state can be decomposed into a sum of countably many measures, each of which is an equilibrium state and is compatible with an inducing scheme as in Theorem~\ref{thm:schemes}.  \cite{BuSar} implies that there is only one equilibrium state per inducing scheme.  Lemma~\ref{lem:press zero} then implies that this equilibrium state must be unique.

We suppose that $\mu$ is an equilibrium state for $-t\log|Df|$ for $t\in (-\infty,t^+)$.  We first note that $\mu$ may be expressed in terms of its ergodic decomposition, see for example \cite[Section 2.3]{Kellbook},  $\mu(\cdot)=\int\mu_y(\cdot)~d\mu(y)$ where $y\in I$ is a generic point of the ergodic measure $\mu_y\in \M_f$.  Clearly, for any set $A\subset I$ such that $\mu(A)>0$, the measure $\mu_A(\cdot):=\frac1{\mu(A)}\int_A\mu_y(\cdot)~d\mu(y)$ must have
$$h(\mu_A)-t\lambda(\mu_A)=p(t),$$
i.e. it must be an equilibrium state itself (otherwise, removing $\mu_A$ from the integral for $\mu$ would increase $h_\mu-t\lambda(\mu)$).
As in the proof of Lemma~\ref{lem:pos ent}, $\lambda(\mu_A)>0$.

Theorem~\ref{thm:schemes}(a) implies that any such $\mu_A$ must decompose into a sum $\mu=\sum_n\alpha_n\mu_n$ where $\mu_n$ is a probability measure compatible with the scheme $(X^n,F_n)$ and $\alpha_n\in (0,1]$.  Then there are $F_n$-invariant probability measures $\mu_{F_n}$, each of which projects to $\mu_n$ by \eqref{eq:lift}.

By Lemma~\ref{lem:press zero} and \cite{BuSar}, $\mu_{F_n}$ must be the unique equilibrium state for the scheme $(X^n, F_n, \tau_n)$ with potential $-t\log|DF_n|-p(t)\tau_n$.  Therefore, $\mu_n$ is the only equilibrium state for $-t\log|Df|$ which is compatible with $(X^n,F_n)$.

We finish the proof by using Theorem~\ref{thm:schemes} b) which implies that any of these equilibrium states compatible with an inducing scheme $(X^n,F_n)$ as above must be compatible with each of the other inducing schemes $(X_j,F_j)$.  Hence $\mu_i=\mu_j$ for every $i,j\in \N$.  Since $\mu$ was an arbitrary equilibrium state, this argument implies that $\mu$ is ergodic and is the unique equilibrium state for $-t\log|Df|$, as required.

Suppose that $t^+>1$.  Since $\lambda_m\ge 0$ this means that $t\mapsto p(t)$ must be strictly decreasing in the interval $(1, t^+)$.  Since Bowen's formula implies that $p(t)\le 0$ this means that $p(t)<0$.  Ruelle's formula \cite{Ruelleineq} then implies that we must have $\lambda_m>0$.  Therefore, if $t^+>1$ then $\lambda(\mu)>0$ for all $\mu\in \M_f$ and so we can apply Theorem~\ref{thm:schemes} to the case $t\ge t^+$ also.
\end{proof}

\section{Proof of Theorem~\ref{thm:eq_exist_unique}}
\label{sec:main thm}

The previous sections give most of the information we need to prove
Theorem~\ref{thm:eq_exist_unique}.  In this section we prove the remaining part: that the critical parameter $t^-$, defined in equation \eqref{eq:t plus minus}, is not finite.  We then put the proof of Theorem~\ref{thm:eq_exist_unique} together.

\begin{lema}
There exists a measure $\mu_M$ such that $\lambda(\mu_M)=\lambda_M$.
\label{lem:max meas}
\end{lema}

\begin{proof}
This follows from the compactness of $\M_f$ and the upper semicontinuity of $x\mapsto \log|Df(x)|$.
\end{proof}

\begin{prop}
$t^-=-\infty$.
\label{prop:t^-}
\end{prop}

\begin{proof}
Suppose, for a contradiction, that $t^->-\infty$.  This implies that for $t\le t^-$, the measure $\mu_M$ in Lemma~\ref{lem:max meas} also maximises $h(\mu)-t\lambda(\mu)$ for $\mu\in \M_f$, and must have $h(\mu_M)=0$.

By Theorem~\ref{thm:schemes}, we can choose an inducing scheme $(X,F)$ compatible with $\mu_M$.
\begin{claim}
$P^G(\Psi_{t})=0$ for all $t\le t^-$.
\end{claim}

\begin{proof}
$P^G(\Psi_{t})\le 0$ follows by Lemma~\ref{lem:press less zero}.  $P^G(\Psi_{t})\ge 0$ follows since $\mu_M$ is compatible with our scheme.
\end{proof}

Since by construction, $\mu_M$ is compatible with $(X,F)$, the induced measure being denoted by $\mu_{F,M}$, and since $h(\mu_M)+\int\psi_t~d\mu_M=0$, we have $$h(\mu_{F,M})+\int\Psi_t~\mu_{F,M}=0,$$
and so $\mu_{F,M}$ is an equilibrium state for $\Psi_t$.  However, by Theorems 1.1 and 1.2 of \cite{BuSar} any equilibrium state of $\Psi_t$ must have positive entropy, a contradiction.
\end{proof}

\begin{proof}[Proof of Theorem~\ref{thm:eq_exist_unique}]
The existence of the equilibrium state for $-t\log|Df|$ and $t\in (t^-, t^+)$ follows from Proposition~\ref{prop:conv to eq}.  Uniqueness follows from Proposition~\ref{prop:unique}.  Positivity of the entropy of $\mu_t$ comes from Lemma~\ref{lem:pos ent}.  Finally the fact that $t^-=-\infty$ comes from Proposition~\ref{prop:t^-}.
\end{proof}

\section{The pressure is of class $C^1$ and strictly convex in $(-\infty,  t^+)$} \label{sec:smooth and convex}

As discussed in the introduction, for general systems the pressure function $t \mapsto p(t)$ is convex, therefore it can have at most a countable number of first order phase transitions.
In \cite{Saran} an example is constructed with the property that the set of parameters at which the pressure function is not analytic has positive measure (in this case, there also exist \emph{higher order phase transitions}, see \cite{Sarcrit}). Nevertheless, for multimodal maps it has been shown that  in certain intervals the pressure function is indeed real analytic, see \cite{BTeqgen,BTeqnat}. Dobbs \cite[Proposition 9]{Dobphase} proved that in the quadratic family $x\mapsto \gamma x(1-x)$, $\gamma \in (3,4)$  there exists uncountably many parameters for which the pressure function admits infinitely many phase transitions.  However, these transitions are caused by the existence of an infinite sequence renormalisations of the map, so for these parameters the corresponding quadratic maps do not have a representative in the class $\F$.  He also notes \cite[Proposition 4]{Dobphase} that in the quadratic family there is a always a phase transition for negative $t$ caused by the repelling fixed point at 0.  Since, this fixed point is not in the transitive part of the system (which actually must be contained in $[f^2(c), f(c)]$), from our perspective this point is not dynamically relevant, so any representative of such a map in $\F$ would miss this part of the dynamics, and hence not exhibit this transition.

\begin{prop} \label{prop:regular}
Let $f \in \F$, the pressure function $p$ is $C^1$ in the interval $(-\infty, t^+)$.
\end{prop}

\begin{proof}
We first show that $p$ is differentiable.  By Theorem~\ref{thm:schemes}, we can choose an inducing scheme $(X,F,\tau)$ which is compatible with $\mu_t$ for each $t\in (-\infty, t^+)$.  Then we have the limits
$$\lim_{t'\to t}\int\log|DF|~d\mu_{\Psi_{t'}}= \int\log|DF|~d\mu_{\Psi_{t}} \text{ and } \lim_{t'\to t}\int\tau~d\mu_{\Psi_{t'}}= \int\tau~d\mu_{\Psi_{t}}.$$  We emphasise that these limits are the same if $t'$ are taken to the left or to the right of $t$.  Hence $\lambda(\mu_{\psi_{t}})$ is continuous in $(t^-, t^+)$.  Since the derivative of $p$ is $-\lambda(\mu_{\psi_{t}})$, 
the derivative is continuous,  proving the lemma.
This standard fact can be seen as follows (see also, for example, \cite[Theorem 4.3.5]{Kellbook}): given $\eps>0$, by the definition of pressure the free energy of $\mu_t$ with respect to $\psi_{t+\eps}$ is no more than $p(t+\eps)$.  Similarly the free energy of $\mu_{t+\eps}$ with respect to $\psi_{t}$ is no more than $p(t)$.  Hence
$$\frac{(-(t+\eps)+t)\lambda(\mu_{t+\eps})}\eps \ge \frac{p(t+\eps)-p(t)}\eps \ge \frac{(-(t+\eps)+t)\lambda(\mu_{t})}\eps.$$
So whenever $t\mapsto \lambda(\mu_{t})$ is continuous, $Dp(t) = -\lambda(\mu_{t})$.
\end{proof}

\begin{prop}
\label{prop:strictconv}
For $f\in \F$, $t^+>0$ and the pressure function $p$ is strictly convex in $(-\infty, t^+)$.
\end{prop}


Before proving this proposition, we need two lemmas: the first guarantees that $t^+>0$, while the second will be used to obtain strict convexity of the pressure function (both these facts are in contrast with the quadratic Chebyshev case).

\begin{lema}
For $f\in \F$, $\lambda(\mu_0)>\lambda_m$ where $\mu_0$ is the measure of maximal entropy for $f$.
\label{lem:not affine}
\end{lema}

\begin{proof}
The existence of a (unique) measure of maximal entropy $\mu_0$ is guaranteed by \cite{Hofpw}.  Suppose for a contradiction that the lemma is false and hence $\lambda(\mu_0)=\lambda_m$.  Since when the derivative of $p$ exists at a point $t$, it is equal to $-\lambda(\mu_t)$ (see \cite{Ruellebook} as well as the computation in the proof of Proposition~\ref{prop:regular}) and by convexity, the pressure function must be affine with constant slope $-\lambda_m$. i.e. $p(t)=h_{top}(f)-t\lambda_m$ for $t\in [0, \infty)$.
This implies that $\mu_0$ must be an equilibrium state for the potential $-t\log|Df|$ for every $t\in \R$.  In particular this applies when $t=1$.  Moreover, by Ruelle's inequality \cite{Ruelleineq}, we have $\lambda(\mu_0)>0$, so $\mu_0$ must be an acip by \cite{Led}.  By \cite[Proposition 3.1]{Dobbsvis}, this implies that $f$ has finite postcritical set, which is a contradiction.
\end{proof}

\begin{lema}
For any $\eps>0$ there exists an inducing scheme $(X,F)$ a sequence $i_k \to \infty$ such that the domains $X_{i_k}$ have
$$|X_{i_k}| \ge e^{-(\lambda_m+\eps)\tau_{i_k}}.$$
\label{lem:good schemes}
\end{lema}

\begin{proof}
It is standard to show that for any $\eps>0$, there exists a periodic point $p$ with Lyapunov exponent $\le \lambda_m +\eps/3$, see for example \cite[Lemma 19]{Dobphase}.  We can choose $(X,F)$ as in Theorem~\ref{thm:schemes} so that the orbit of $p$ is disjoint from $X$.  We may further assume that $(X,F)$ has distortion bounded by $e^{\delta}$ for some $\delta>0$, i.e. $$\frac{|DF(x)|}{|DF(y)|} \le e^{\delta}$$ for all $x, y\in X_i$ for any $i\in \N$.  In this case, by the transitivity of $(I,f)$, which is reflected in our inducing scheme, there must exist an infinite sequence of domains $X_{n_k}$ of $(X,F)$ which shadow the orbit of $p$ for longer and longer.  One can use standard distortion arguments to prove that for all large $k$, $|X_i| \ge |X| e^{-\delta} e^{-(\lambda_m+\eps/2)\tau_i}$.  Choosing $\delta>0$ appropriately completes the proof of the lemma.
\end{proof}

\begin{proof}[Proof of Proposition~\ref{prop:strictconv}]
For the first part of the proposition, $t^+>0$ is guaranteed by Lemma~\ref{lem:not affine}.

For the second part of the proposition, since $p$ is convex, we only have to rule out $p$ being affine in some interval.
Suppose first that $p$ is affine in an interval $[t_1, t_2] \subset (-\infty, t^*)$ where $t^*:=\inf\{t:Dp(t) = -\lambda_m\}$. I.e. for some $\beta>\lambda_m$, $t\in [t_1, t_2]$ implies $p(t) = p(-t_1)-(t-t_1)\beta$.  We let $\eps>0$ be such that $\beta>\lambda_m+\eps$.
By Lemma~\ref{lem:good schemes}, there exists $i_k\to \infty$ such that
$$|X_{i_k}| \ge e^{-(\lambda_m+\eps)\tau_{i_k}}.$$

The fact that the pressure function is affine in $[t_1, t_2]$ implies that the equilibrium state is the same $\mu$ for every $t\in [t_1, t_2]$.  Denote the induced version of $\mu$ by $\mu_F$.
By the Gibbs property of our inducing schemes, $\mu_F(X_i) \asymp |X_i|^te^{-\tau_ip(t)}$ for all $t\in [t_1, t_2]$.  Therefore, $$\frac{|X_i|^{t_1} e^{-\tau_ip(t_1)}}{|X_i|^t e^{-\tau_i(p(t_1) -(t-t_1)\beta)}}\asymp 1,$$
which implies that $$|X_i| \asymp e^{-\tau_i\beta}$$ for all $i$.
Since
$$|X_{i_k}| \ge e^{-(\lambda_m+\eps)\tau_{i_k}}$$ for an infinite sequence of domains $X_{i_k}$, and $\beta>\lambda_m+\eps$, this yields a contradiction.

We next want to prove that $t^+ =t^*$.  We suppose not in order to get a contradiction.  In the first case suppose that $\lambda_m=0$.  Then $p(t)\ge 0$ for all $t\in \R$. Coupled with Bowen's formula this implies that $p(1)= 0$.  So the convexity of $p$ implies $t^+=t^*$, as required.  Now suppose that $\lambda_m>0$.  Since we assumed $t^+$ the graph of $p(t)$ must be above, and parallel to $t\mapsto -t\lambda_m$ on $[t^*,\infty)$.  This implies that $t^+=\infty$ and so Theorem~\ref{thm:eq_exist_unique} gives equilibrium states for all $t\in \R$.  Hence we can mimic the argument above, with the inducing scheme as in Theorem~\ref{thm:schemes} compatible with $\mu_{t^*}$, but instead taking $[t_1,t_2] \subset [t^*, \infty)$ and $\beta=\lambda_m$.  Noting that the argument of Lemma~\ref{lem:good schemes} ensures that we chose the scheme $(X,F)$ so that there is a sequence of domains $|X_{i_k}| \le e^{-(\lambda_M-\eps)\tau_{i_k}}$, we can complete the argument.
\end{proof}

\begin{proof}[Proof of Theorem~\ref{thm:smooth}]
The convexity  of $p$ follows from Proposition~\ref{prop:strictconv}, the smoothness from \ref{prop:regular} and the fact that the pressure is decreasing from \cite{Prz}.
\end{proof}

\section{Phase transitions in the positive spectrum}
\label{sec:kinks acips}

In this section we study the relation between the existence of first order phase transitions  at the point $t=1$ and the existence of an acip. The following proposition has Proposition~\ref{prop:acip} as a corollary.

\begin{prop}
Suppose that $f\in \F$ has $\lambda_m=0$.  Then $f$ has an acip if and only if $p$ has a first order phase transition at $t=1$. \label{prop:acip kink}
\end{prop}

\begin{rem}
Note that if $\lambda_m>0$ then the situation is quite different.  For example if $f\in \F$ satisfies the Collet-Eckmann condition (which by \cite{BrvS} implies $\lambda_m>0$), in which case the map also has an acip, then by \cite[Theorem 3]{BTeqnat}, $p$ is real analytic in a neighbourhood of $t=1$.
\end{rem}

The following lemma will be used to prove Proposition~\ref{prop:acip kink}.

\begin{lema}
Suppose that $t^+\in (0, \infty)$ and there is a first order phase transition at $t^+$.  Then there exists an inducing scheme $(X,F)$, an equilibrium state $\mu_{\Psi}$ for $\Psi=\Psi_{t^+}$, and an equilibrium state $\mu_\psi$ for $\psi=\psi_{t^+}$ with $h(\mu_\psi)>0$.
\label{lem:meas at kink}
\end{lema}

\begin{proof}
The fact that there is a first order phase transition implies that the left derivative of $p$ at $t^+$ has $Dp^-(t^+)<-\lambda_m$.
The convexity of the pressure function implies that the graph of the pressure lies above the line $t\mapsto D^-p(t^+)t-t^+(\lambda_m+D^-p(t^+))$.
This means that we can take a sequence of equilibrium states $\mu_{\psi_t}$ for $t$ arbitrarily close to, and less than, $t^+$ with free energy converging to $p(t^+)$ with
$$h(\mu_{\psi_t}) \ge -t^+(Dp^-(t^+)+\lambda_m)>0.$$  Hence the  arguments used to prove Proposition~\ref{prop:conv to eq} give us an equilibrium state for $\mu_{t^+}$ with positive entropy.
\end{proof}

\begin{proof}[Proof of Proposition~\ref{prop:acip kink}]
If there exists an acip $\mu$ then $Dp^-(1)=-\lambda(\mu)<0$.  Since $\lambda_m=0$ implies $p(t)=0$ for all $t\ge 1$, the existence of an acip implies that there is a first order phase transition at $t=1$.

On the other hand, if there exists a first order phase transition at $t=1$ then Lemma~\ref{lem:meas at kink} implies that there is an equilibrium state $\mu_1$ for $-\log|Df|$, with $h(\mu_1)>0$.  By \cite{Led} this must be an acip.
\end{proof}


\begin{rem}
If $\lambda_m>0$ and there is a measure $\mu_m$ such that $\lambda(\mu_m)=\lambda_m$, then by Lemma~\ref{lem:meas at kink} and the arguments in the proof of Proposition~\ref{prop:t^-}, we have that $t^+=\infty$.
\label{rmk:min meas}
\end{rem}

\begin{rem}
There are examples of maps in $\F$ with $\{\mu\in \M_f:\lambda(\mu)=\lambda_m\} = \es$, for example \cite[Lemma 5.5]{BrKell}, a quadratic map in $\F$ is defined so that $\lambda_m=0$, but there are no measures with zero Lyapunov exponent.  There are also examples of maps $f\in \F$ with $\{\mu\in \M_f:\lambda(\mu)=\lambda_m\} \neq \es$, for example in \cite{Brminim} examples of quadratic maps in $\F$ are given for which the omega-limit set of the critical point supports (multiple) ergodic measures with zero Lyapunov exponent. Moreover, Cortez and Rivera-Letelier \cite{CJ}  proved that given $\mathcal{E}$ a non-empty, compact, metrisable and totally disconnected topological space then there exists a parameter $\gamma \in (0,4]$ such that set of invariant probability measures of $x\mapsto \gamma x(1-x)$, supported on the omega-limit set of the critical point is homeomorphic to $\mathcal{E}$.
\end{rem}


It is plausible that there are maps $f\in \F$ for which
\begin{equation*}
\inf \left\{ t \in \mathbb{R} : p(t) \le 0 \right\} <1.
\end{equation*}
However, the following argument shows that this is not true for unimodal maps with quadratic critical point in $\F$.

Given an interval map $f:I \to I$, we say that $A\subset I$ is a \emph{metric attractor} if $B(A):=\{\omega(x)\subset A\}$ has positive Lebesgue measure and there is no proper subset of $A$ with this property.  On the other hand $A$ is a \emph{topological attractor} if $B(A)$ is residual and there is no proper subset of $A$ with this property.  We say that $f$ has a \emph{wild attractor} if there is a set $A$ which is a metric attractor, but not a topological one.

\begin{prop}
\label{prop:pos press}
If $f\in \F$ is a unimodal map with no wild attractor then for $t<1$, $p(t)>0$.
\end{prop}

\begin{rem} It was shown in \cite{BKNS} that there are  unimodal maps with wild attractors in $\F$.  However, if $\ell_c=2$ then this is not possible by \cite{Lyuwild}.
\label{rmk:wild}
\end{rem}

\begin{lema}\label{lem:start}
If $f\in \F$ is a unimodal map with no wild attractor then for each $\eps>0$ there exists a measure $\mu\in \M_f$ so that $$\frac{h(\mu)}{\lambda(\mu)}>1-\eps.$$
\end{lema}

\begin{proof}
By \cite[Theorem V.1.4]{MSbook}, originally proved by Martens, there must be an inducing scheme $(X,F)$ such that $Leb(X\sm \cup_i X_i)=0$.  For any $\delta>0$ we can truncate $(X,F)$ to a finite scheme $(X^N, F_N)$ where $X^N =\cup_{i=1}^N X_i$ so that $Leb(\cup_{i=1}^N X_i) > (1-\delta)|X|.$
We therefore have
$$\dim_H\left\{x:\tau^k(x)<\infty\text{ for all } k\in \N\right\}>1-\delta'$$ where $\delta'$ depends on $\delta$ and the distortion of $F$ (in particular $\to 0$ as $\delta\to 0$).
It follows from the Variational Principle and the Bowen formula (see \cite[Chapter 7]{pe}) that  there is an $F$-invariant measure, $\mu_F$, for this system with $$\frac{h(\mu_F)}{\lambda(\mu_F)} >1-\delta'.$$
By the Abramov formula, for $\mu$ the projection of $\mu_F$,
$$\frac{h(\mu)}{\lambda(\mu)}>1 -\delta'$$
also. Choosing $\delta>0$ so small that $\delta'\le\eps$ completes the proof.
\end{proof}

\begin{proof}[Proof of Proposition~\ref{prop:pos press}]
Let $t<1$ and choose $\eps=1-t>0$.  Then the measure $\mu$ in Lemma~\ref{lem:start} has $h(\mu)-t\lambda(\mu)>0$.  Hence by the definition of pressure, $p(t)>0$.
\end{proof}

\begin{proof}[Proof of Proposition~\ref{prop:collected results}]
By Proposition~\ref{prop:pos press} and Remark~\ref{rmk:wild} we can take $t^+=1$.  Hence we can conclude that $p$ is $C^1$ strictly convex decreasing in $(-\infty, t^+)$ by Theorem~\ref{thm:smooth}.  The fact that $p(t)=0$ for all $t\ge 1$ follows from \cite{NoSa}.

Part (a) follows from Proposition~\ref{prop:acip kink} since this implies that both left and right derivatives of $p(t)$ at $t=1$ are zero.  Part (b) is the converse of this since the left derivative is strictly negative and the right derivative is zero.
\end{proof}

\section{Remarks on statistical properties and Chebyshev polynomials} \label{sec:remarks}

In this section we collect some further comments on our results.

\subsection{Statistical properties}
Given $f\in \F$ and an equilibrium state $\mu$ as in Theorem~\ref{thm:eq_exist_unique}, one can ask about the statistical properties of the system $(I,f,\mu)$.  For general equilibrium states we expect it should be possible to prove exponential decay of correlations along with many other statistical laws.   From the theory presented here allied to \cite{BTret}, one can show that the system $(I, f,\mu)$ has `exponential return time statistics' (for definitions see for example \cite{BTret}).

\subsection{Ergodic Optimisation}

Let $f \in \F$ and  $\phi: [0,1] \to \mathbb{R}$ a function. The study of invariant probability measures whose ergodic
$\phi-$average is as large (or as small) as possible is known as
\emph{ergodic optimisation}.  A measure $\mu \in \M_f$ is called $\phi-$minimising/maximising if
\[\int \phi \ d\mu = \inf \left\{ \int \phi \ d\nu : \nu \in \M_f \right\} \text{ or } \int \phi \ d\mu = \sup \left\{ \int \phi \ d\nu : \nu \in \M_f \right\} \]
respectively.
For a survey on the subject see \cite{OJ}.  Let $t\in (-\infty, t^+)$ and denote by $\mu_t$ the unique equilibrium state corresponding to  the potential $-t\log|Df|$. A consequence of the results in this paper is that: any accumulation point $\mu$ of a sequence of measures $\mu_{t_n}$, given by Theorem~\ref{thm:eq_exist_unique}, where $t_n \to -\infty$ is a $\log|Df|$-maximising measure.  This is because $\log|Df|$ is upper semicontinuous; $Dp(t)=-\lambda(\mu_t)$; and this derivative is asymptotic to $-\lambda_M$.  Hence there is a subsequence of these measures $(\mu_{t_{n_k}})_k$ so that $$\lim_{k\to \infty} \lambda(\mu_{t_{n_k}})=\lambda(\mu)=\lambda_M.$$

Note that Lemma~\ref{lem:max meas} guarantees the existence of a $\log |Df|-$maximising measure.  (We do not assert anything about the uniqueness of this measure.) Actually, any measure $\mu$, which is an accumulation point of $\mu_{t_n}$ as  $t_n \to -\infty$, is a measure maximising entropy among all measures which maximise $\log |Df|$.  Then in fact $p(t)$ is asymptotic to the line $h(\mu)-t\lambda_M$ as $t\to -\infty$.

\subsection{The preperiodic critical point case}
\label{sec:preper}

For our class of maps $\F$ we assumed that the orbit of points in $\crit$ are infinite.   Here we comment on an alternative case.
In the case of the quadratic Chebychev polynomial $x\mapsto 4x(1-x)$ on $I$, it is well known that the two relevant measures are the acip $\mu_1$, which has $\lambda(\mu_1)=\log 2=\lambda_m$, and the Dirac measure $\delta_0$ on the fixed point at 0, which has $\lambda(\delta_0)= \log 4=\lambda_M$.  So $t^-=-1$ and
\begin{equation*}
p(t)= \begin{cases} (1-t)\log 2& \text{if } t\ge -1,\\
 -t\log 4&  \text{if } t \le -1.
\end{cases}
\end{equation*}

Note that the above piecewise affine form for the pressure function does not conflict with Theorem~\ref{thm:smooth}, which might be expected to apply in the interval $(t^-, t^+)$, since $t^+=t^*=-1$, where $t^*$ is defined in the proof of Proposition~\ref{prop:strictconv}.

\appendix

\section{Cusp maps}

In this section we outline how to extend the above results to some maps which are not smooth.  This class includes the class of contracting Lorenz-like maps, see for example \cite{Rovella}.

\begin{defi}
$f:\cup_jI_j\to I$ is a \emph{cusp map} if there exist constants $C,\alpha>1$ and a set $\{I_j\}_j$ is a finite  collection of disjoint open subintervals of $I$ such that
\begin{enumerate}
\item $f_j:=f|_{I_j}$ is $C^{1+\alpha}$ on each $I_j=:(a_j, b_j)$ and $|Df_j|\in (0, \infty)$.

\item $D^+f(a_j), \ D^-f(b_j)$ exist and are equal to 0 or $\pm\infty$.

\item For all $x,y\in \overline{I_j}$ such that $0<|Df_j(x)|, |Df_j(y)|\le 2$ we have $|Df_j(x)-Df_j(y)|<C|x-y|^\alpha$.

\item For all $x,y\in \overline{I_j}$ such that $|Df_j(x)|, |Df_j(y)|\ge 2$, we have $|Df_j^{-1}(x)-Df_j^{-1}(y)|<C|x-y|^\alpha$.
\end{enumerate}
We denote the set of points $a_j, b_j$ by $\crit$.
\end{defi}

\begin{rem}
Notice that if for some $j$, $b_j=a_{j+1}$, i.e. $I_j\cap I_{j+1}$ intersect, then $f$ may not continuously extend to a well defined function at the intersection point $b_j$, since the definition above would then allow $f$ to take either one or two values there.  So in the definition above, the value of $f_j(a_j)$ is taken to be $\lim_{x\searrow a_j}f_j(x)$ and $f_j(b_j)=\lim_{x\nearrow b_j}f_j(x)$, so for each $j$, $f_j$ is well defined on $\overline{I_j}$.
\label{rmk:boundaries}
\end{rem}

\begin{rem}
In contrast to the class of smooth maps $\F$ considered previously in this paper, for cusp maps we can have $\lambda_M=\infty$ and/or $\lambda_m=-\infty$.  The first possibility follows since we allow singularities (points where the one-sided derivative is $\infty$).  The second possibility follows from the presence of critical points (although is avoided for smooth multimodal maps with non-flat critical points by \cite{Prz}).  Examples of both of these possibilities can be found in \cite[Section 3.4]{Dobthes}.
\label{rmk:lyap infinity}
\end{rem}


We will ultimately be interested in cusp maps without singular points with negative Schwarzian derivative (in fact the latter rules out the former).  Note that since we are only interested in the transitive parts the system, transitive multimodal maps as in the rest of the paper can be considered to fit into this class.

We show below that we can build a Hofbauer extension $(\hat I, \hat f)$.  
We note that the possible issue of $f$ not being well defined at the boundaries of $I_j$, discussed in Remark~\ref{rmk:boundaries}, does not change anything in the definition of the Hofbauer tower.


We next define the Hofbauer extension.  The setup we present here can be applied to general dynamical systems, since it only uses the structure of dynamically defined cylinders.  An alternative way of thinking of the Hofbauer extension specifically for the case of multimodal interval maps, which explicitly makes use of the critical set, is presented in \cite{BrBr}.

We let $\cyl_n[x]$ denote the member of $\P_n$, which defined as above, containing $x$.  If $x\in \cup_{n\ge 0}f^{-n}(\crit)$ there may be more than one such interval, but this ambiguity will not cause us any problems here.

The \emph{Hofbauer extension} is defined as $$\hat
I:=\bigsqcup_{k\ge 0}\bigsqcup_{\cyl_{k}\in \P_{k}}
f^k(\cyl_{k})/\sim$$ where $f^k(\cyl_{k})\sim
f^{k'}(\cyl_{k'})$ as components of the disjoint union $\hat I$ if $f^k(\cyl_{k})= f^{k'}(\cyl_{k'})$ as subsets in $I$.  Let
$\D$ be the collection of domains of $\hat I$ and $\pi:\hat
I \to I$ be the natural inclusion map.  A point $\hat x\in \hat I$ can
be represented by $(x,D)$ where $\hat x\in D$ for $D\in \D$ and
$x=\pi(\hat x)$.  Given $\hat x\in \hat I$, we can denote the domain $D\in \D$ it belongs to by $D_{\hat x}$.

The map $\hat f:\hat I \to \hat I$ is defined by
$$\hat f(\hat x) = \hat f(x,D) = (f(x), D')$$
if there are cylinder sets $\cyl_k \supset \cyl_{k+1}$ \st $x \in
f^k(\cyl_{k+1}) \subset f^k(\cyl_{k}) = D$ and $D' = f^{k+1}
(\cyl_{k+1})$.
In this case, we write $D \to D'$, giving $(\D, \to)$ the
structure of a directed graph.  Therefore, the map $\pi$
acts as a semiconjugacy between $\hat f$ and $f$: $$\pi\circ \hat
f=f\circ \pi.$$
We denote the `base' of $\hat I$, the copy of $I$ in $\hat I$, by  $D_0$.  For $D\in \D$, we define $\level(D)$ to be the length of the shortest path $D_0 \to \dots \to D$ starting at the base $D_0$.  For each $R \in \N$, let $\hat I_R$ be the compact
part of the Hofbauer tower defined by
$$
\hat I_R := \sqcup \{ D \in \D : \level(D) \le R \}.$$

For maps in $\F$, we can say more about the graph structure of $(\D, \to)$ since Lemma 1 of \cite{BTeqnat} implies that if $f\in \F$ then there is a closed primitive subgraph $\D_{\T}$ of $\D$.  That is, for any $D,D' \in\D_{\T}$ there is a path $D\to \cdots \to D'$; and for any $D\in \D_{\T}$, if there is a path $D\to D'$ then $D'\in \D_{\T}$ too.  We can denote the disjoint union of these domains by $\hat I_{\T}$.  The same lemma says that if $f\in \F$ then $\pi(\hat I_{\T})=\Omega$, the non-wandering set and $\hat f$ is transitive on $\hat I_{\T}$.  Theorem~\ref{thm:cusp facts} gives these properties for transitive cusp maps.

Given an ergodic measure $\mu\in \M_f$, we say that $\mu$ \emph{lifts to $\hat I$} if there exists an ergodic $\hat f$-invariant probability measure $\hat\mu$ on $\hat I$ such that $\hat\mu\circ\pi^{-1}=\mu$.  For $f\in \F$, if $\mu\in \M_f$ is ergodic and $\lambda(\mu)>0$ then $\mu$ lifts to $\hat I$, see \cite{Kellift, BrKell}.

Property $(*)$ is that for any $\hat x, \hat y\notin \bd\hat I$ with $\pi(x)=\pi(y)$ there exists $n$ such that $\hat f^n(\hat x)=\hat f^n(\hat y)$.  This follows for cusp maps by the construction of $\hat I$ using the branch partition.

We will only use the following result in the context of equilibrium states for cusp maps with no singularities.  However, for interest we state the theorem in greater generality.

\begin{teo}
Suppose that $f:I \to I$ is a transitive cusp map with $h_{top}(f)>0$.  Then:
\begin{enumerate}
\item there is a transitive part $\hat I_{\T}$ of the tower such that $\pi(\hat I_{\T})=I$;
\item any measure $\mu\in \M_f$ with $0<\lambda(\mu)<\infty$ lifts to $\hat\mu$ with $\mu=\hat\mu\circ \pi^{-1}$;
\item for each $\eps>0$ there exists $\eta>0$ and a compact set $\hat K\subset \hat I_{\T}\sm \bd \hat I$ such that any measure $\mu\in \M_f$ with $h(\mu)>\eps$ and $0<\lambda(\mu)<\infty$ has $\hat\mu(\hat K)>\eta$.
\end{enumerate}
\label{thm:cusp facts}
\end{teo}

\begin{proof}
\text{Part (1):}
The first part can be shown as in \cite[Lemma 2]{BTeqnat}, but we argue as in \cite{Hofpw} (see also \cite[Theorem 6]{KellStP}).  Theorem 11 of that paper gives a decomposition of $\hat I$ into a countable union $\Gamma$ of irreducible (maximal with these properties) closed (if there is a path $D\to D'$ for $D\in \E$ then $D'\in \E$) primitive (there is a path between any $D,D'\in \E$) subgraphs $\E$ along with some sets which carry no entropy.  Since $h_{top}(\hat f)=h_{top}(f)$ and we have positive topological entropy, this means that $\Gamma \neq \es$.  Let $\E\in \Gamma$.  Clearly $\pi(\E)$ is open, so by the transitivity of $f$, there must be a point $x\in \pi(\E)$ which has a dense orbit in $I$.  By definition, $\omega(x)\subset\pi(\E)$.
By property $(*)$, $\pi(\E)\cap \pi(\E') = \es$ for any $\E,\E'\in \Gamma$ which implies that $\#\Gamma =1$.  That there is a dense orbit in $(\E,\hat f)$ follows from the Markov property of this subgraph, so we let $\hat I_{\T}=\sqcup_{D\in \E}D$.

\text{Part (2):}
Ledrappier, in \cite[Propostion 3.2]{Led} proved the existence of non-trivial local unstable manifolds for a more general class of maps (so-called PC-maps) with an ergodic measure $\mu\in \M_f$ with $\lambda(\mu)>0$.  However, he also required a non-degeneracy condition.  For cusp maps, Dobbs \cite[Theorem 13]{Dobcusp} was able to this but without the non-degeneracy requirement.

Keller showed in \cite[Theorem 6]{Kellift} that the existence of such unstable manifolds means that any non-atomic ergodic measure $\mu\in \M_f$ with $\lambda(\mu)>0$ lifts to $\hat\mu$ on $(\hat I, \hat f)$ and that $\mu= \hat\mu\circ\pi^{-1}$.  Using Dobbs and assuming that $\mu$ is not supported on $\cup_{n\ge 0}f^n(\crit)$ we can drop the non-atomic assumption (see also \cite[Theorem 3.6]{BrKell}).

\text{Part (3):}
The third part follows exactly as in \cite[Lemma 4]{BTeqnat}.
\end{proof}

Suppose now that $f$ is a cusp map without singularities (i.e. $|Df|$ is bounded above), with negative Schwarzian and such that the non-wandering set $\Omega$ is an interval.  We consider $f:\Omega \to \Omega$.
For each $t\in (t^-, t^+)$, we can find a finite number of inducing schemes as in Proposition~\ref{thm:schemes} with which all measures with large enough free energy w.r.t. $\psi_t$ will be compatible.  It is important here that we assume negative Schwarzian derivative since we need bounded distortion for our inducing schemes.  This then allows us to prove Theorem~\ref{thm:eq_exist_unique} for this class of maps, but we may have $t^->-\infty$.  If we exclude maps with preperiodic critical points then we again have $t^-=-\infty$.  Similarly we can prove Theorem~\ref{thm:smooth} for this class of maps, although again we only get $t^-=-\infty$ if we exclude maps with preperiodic critical points.  Note also that the fact that $\lambda_m$ can be negative, and may even be $-\infty$, implies that $t^+$, which for the class $\F$ had to lie in $[1,\infty]$, could be any value in the range $[0,\infty]$ for cusp maps.

Note that for the maps considered by Rovella in \cite{Rovella}, the critical values are periodic and so the measure supported on them is not seen by our inducing schemes.  This is like the Chebyshev case, so as in that situation, the pressure function could be piecewise affine.

 \end{document}